\documentclass[final]{siamltex}
\usepackage{makeidx}  
\usepackage{amsfonts}
\usepackage{amsmath}
\usepackage{graphicx}
\usepackage{latexsym}
\usepackage{amssymb}
\usepackage[lined,boxed]{algorithm2e}

\newtheorem{example}{Example}[section]
\newtheorem{problem}{Problem}

\title{A control strategy algorithm for finite alternating transition systems \thanks{This work received financial
support of the National Natural Science of China (No. 60973045), the NSF of Jiangsu Province
(No. BK2007191) and Fok Ying-Tung Education Foundation.}}
\author{Jinjin Zhang\thanks{Department of Computer Science, Nanjing University of Aeronautics and Astronautics, Nanjing, P. R. China, 210016 ({\tt jinjinzhang@nuaa.edu.cn}).}
        \and Zhaohui Zhu\thanks{Corresponding author. Department of Computer Science, Nanjing University of Aeronautics and Astronautics, Nanjing, P. R. China, 210016; State Key Lab of Novel Software Technology,
Nanjing University,
Nanjing, P. R. China, 210093 ({\tt bnj4892856@jlonline.com}.)        }
        \and Jianfei Yang\thanks{Department of Automation Engineering,
Nanjing University of Aeronautics and Astronautics,
Nanjing, P. R. China, 210016 ({\tt
yjfsmile@nuaa.edu.cn})}}
\begin{document}
%
%
%
%
\maketitle              

\begin{abstract}
Recently, there has been an increasing interest in the formal analysis and design of control systems.
In this area, in order to reduce the complexity and scale of control systems, finite abstractions of control systems are introduced and explored.
Amongst, Pola and Tabuada construct finite alternating transition systems as approximate finite abstractions for control systems with disturbance inputs [\textit{SIAM Journal on Control and Optimization}, Vol. 48, 2009, 719-733].
Given linear temporal logical formulas as specifications, this paper provides a control strategy algorithm to find control strategies of Pola and Tabuada's abstractions enforcing specifications.
\end{abstract}

\pagestyle{myheadings}
\thispagestyle{plain}
\markboth{JINJIN ZHANG, ZHAOHUI ZHU, AND JIANFEI YANG}{CONTROL STRATEGY ALGORITHM FOR ALTERNATING SYSTEMS}
\section{Introduction}\label{Sec:introduction}
The formal analysis and design of control systems is one of recent trends in control theory.
The formal analysis is concerned with verifying whether a control system satisfies a desired specification, while the purpose of the formal design is to construct a controller for control system so that it meets a given specification.
Traditionally, stability and reachability are considered as specifications in the control-theoretic community~\cite{hab:1,hab:2}. Recently, there has been an increasing interest in extending the formal analysis and design by considering more complex specifications \cite{alur,anton:1,fain:1,fain:2,kloe:1,Koutsoukos,lacer:1,tab:1,tab:2}.
In these work, temporal logic \cite{alur,anton:1,fain:1,fain:2,kloe:1,tab:1}, regular expressions \cite{Koutsoukos}, and transition systems \cite{tab:2} are used to describe specifications. Amongst, temporal logic, due to its resemblance to natural language and the existence of algorithms for model checking, is widely adopted for task specification and controller synthesis in control theory. For example, linear temporal logic (LTL) has been adopted to describe the desired properties of discrete-time linear systems~\cite{tab:1} and continuous-time linear systems \cite{kloe:1}. In addition, Computation Tree Logic (CTL)\cite{anton:1} and LTL\cite{fain:1,fain:2} are applied to express specifications in the area of mobile robotics.

The formal analysis and design of large-scale control systems is difficult because of the complexity and scale of systems.
In order to reduce the complexity and scale, finite abstractions are extracted from  these control systems~\cite{alur,tab:1,tab:2}.
Usually, finite abstractions and original systems share properties of interest and the analysis and design of finite abstractions is simpler than that of original control systems. Thus the analysis and design of control systems is often equivalently performed on the corresponding  finite abstractions.
So finite abstractions are extremely useful in the formal analysis and design.

Much work has been devoted to the construction of finite abstractions of control systems. For instance, Tabuada and Pappas identify critical properties of discrete-time linear systems ensuring the existence of finite abstractions \cite{tab:3}. Symbolic models of nonlinear control systems are constructed in \cite{pola:4,tab:4}.
Finite abstractions of hybrid systems are studied in~\cite{alur:2,alur:3,henz,henz:1,laffer}. An excellent review of these work may be found in~\cite{alur}.

In the work mentioned above, researchers consider control systems without reference to disturbances.
However, as pointed out by B C. Kuo in \cite{Kuo}, all physical systems are subject to some types of extraneous disturbances or noise during operation.
Recently, Pola and Tabuada extend the above work to control systems affected by disturbances~\cite{pola:5,pola:2}.
A mathematical structure called alternating transition system is presented as symbolic abstraction  of control system with disturbance inputs~\cite{pola:5,pola:2}.
Under the assumption that control systems are bounded, such abstractions are finite.

In~\cite{fain:2}\cite{tab:1}\cite{tab:2}, usual transition systems are adopted as finite abstractions of control systems.
Some approaches are presented to construct control strategies of these finite abstractions enforcing specifications.
Further, based on such control strategies, controllers of original control systems are generated to meet specifications.
So  the construction of control strategies of finite abstractions is one of the important steps in the formal design of control systems.
However, since Pola and Tabuada's abstractions~\cite{pola:5,pola:2} are modeled by alternating transition systems rather than usual transition systems, the approaches provided in~\cite{fain:2}\cite{tab:1}\cite{tab:2} are not suitable for establishing control strategies for Pola and Tabuada's abstractions.
To overcome this defect, this paper will present a control strategy algorithm based on Kabanza et al.'s planning algorithm~\cite{kab:1} to solve the following control problem:
given a finite, non-blocking alternating transition system $T$ and a specification, how to find an initial state and a control strategy of $T$ enforcing the given specification?
Clearly, this algorithm can be used to find control strategies for Pola and Tabuada's finite abstractions.

The rest of this paper is organized as follows. In Section 2, we recall the notion of alternating transition system and present the control problem mentioned above in detail.
Section~3 recalls some notions and results about Kabanza et al.'s planning algorithm.
 Based on their algorithm, Section~4 provides a control strategy algorithm.
In Section~5, we explore the correctness and completeness of this algorithm.
Finally, we conclude the paper with future work in Section~6.
The appendix includes the proofs of some results of this paper.
\section{Alternating transition system and control problem}\label{Sec:pre}
Before recalling the notion of alternating transition system, we introduce some useful notations. The symbol $\mathbb{N}$ denotes the set of positive integers.
For any set $A$, $A^{+}$ denotes the set of all non-empty finite strings over $A$, and $A^{\omega}$ represents the set of infinite strings over $A$. Usually, we put $A^{\infty}=A^{+}\cup A^{\omega}$. We use $s_{A}$, $\sigma_{A}$ and $\alpha_{A}$ to denote the elements of $A^{+}$, $A^{\omega}$ and $A^{\infty}$, respectively. If $A$ is known from the context, we will omit the subscript in $s_{A}$, $\sigma_{A}$ and $\alpha_{A}$.
For any $s\in A^{+}$, $s[i]$ and $s[end]$ mean the $i$-th element and the last element of $s$, respectively.
Given $i\leq j$, $s[i,j]$, $s[i,end]$ and $\sigma[i,\infty]$ represent $s[i]s[i+1]\cdots s[j]$, $s[i]s[i+1]\cdots s[end]$ and $\sigma[i]\sigma[i+1]\cdots$, respectively. As usual, $|s|$ means the length of $s$. For any $\sigma\in A^{\omega}$, $|\sigma|$ is set to be $\infty$.

Pola and Tabuada provide finite abstractions for control systems with disturbance inputs. For these control systems, the inputs consist of control and disturbance inputs, where the former are controllable and the latter are not.
Usual transition system can not capture the different roles played by these two kinds of inputs.
To overcome this obstacle,
Pola and Tabuada adopt alternating transition systems as models of these control systems and their abstract systems~\cite{pola:5,pola:2}.
\begin{definition}\label{def:system}
An alternating transition system is a tuple:
\begin{center}
$T=(Q,A,B,\longrightarrow,O,H)$,
\end{center}
consisting of

$\bullet$ a set of states $Q$;

$\bullet$ a set of control labels $A$;

$\bullet$ a set of disturbance labels $B$;

$\bullet$ a transition relation $\rightarrow\subseteq Q\times A\times B\times Q$;

$\bullet$ an observation set $O$;

$\bullet$ an observation function $H:Q\rightarrow O$.

An alternating transition system is said to be

$\bullet$ finite if $Q$, $A$ and $B$ are finite;

$\bullet$ non-blocking if $\{q':q\xrightarrow{a,b}q'\}\not=\emptyset$ for any $q\in Q$, $a\in A$ and $b\in B$.

An infinite sequence $\sigma\in Q^{\omega}$ is said to be a trajectory of $T$ if and only if for all $i\in\mathbb{N}$, $\sigma[i]\xrightarrow{a_{i},b_{i}} \sigma[i+1]$ for some $a_{i}\in A $ and $b_{i}\in B$.
\end{definition}

 In the above definition, a transition label is a pair $<a,b>$, where the former is used to denote control input and the latter represents disturbance input.
Pola and Tabuada construct non-blocking alternating transition systems as abstractions of control systems with disturbance inputs~\cite{pola:5,pola:2}.
Under the assumption that control systems are bounded, their abstractions are finite.
The related notions and results can be found in~\cite{pola:5,pola:2}.

This paper aims to provide an approach to obtain control strategies of Pola and Tabuada's finite abstractions to meet specifications.
Formally, we will solve the following control problem:
\begin{problem}\label{problem1}
Given a finite, non-blocking alternating transition system $T$ and a specification, how to find an initial state and a control strategy of $T$ enforcing the given specification?
\end{problem}

In this paper, the specifications mentioned above will be described by the linear temporal logic LTL$_{-X}$~\cite{emerson}.
The LTL$_{-X}$ formulae have been used to specify the desired properties of control system and its abstraction in~\cite{kloe:1}. We recall this logic below.

\begin{definition}\cite{emerson,kloe:1}\label{def:ltl}
Let $\mathbb{P}$ be a finite set of atomic propositions. The linear temporal logic LTL$_{-X}(\mathbb{P})$ formula over $\mathbb{P}$ is inductively defined as:
\begin{center}
$\varphi::=p|\neg \varphi|\varphi_{1}\wedge\varphi_{2}|\varphi_{1}\mathbf{U}\varphi_{2}$
\end{center}
where $p\in \mathbb{P}$.
\end{definition}

The operator $\mathbf{U}$ is read as ``until'' and the formula $\varphi_{1}\mathbf{U}\varphi_{2}$ specifies that $\varphi_{1}$ must hold until $\varphi_{2}$ holds.
The semantics of LTL$_{-X}(\mathbb{P})$ formulae are defined below.

\begin{definition}\label{Def:satisfaction}
Let $\sigma_{\mathbb{P}}$ be any infinite word over $2^{\mathbb{P}}$ (i.e.,$\sigma_{\mathbb{P}}\in (2^{\mathbb{P}})^{\omega}$).
 The satisfaction of LTL$_{-X}(\mathbb{P})$ formula $\varphi$ at position $i\in \mathbb{N}$ of the word $\sigma_{\mathbb{P}}$, denoted by $\sigma_{\mathbb{P}}[i]\models\varphi$, is defined inductively as follows:

(1) $\sigma_{\mathbb{P}}[i]\models p$ iff $p\in \sigma_{\mathbb{P}}[i]$;

(2) $\sigma_{\mathbb{P}}[i]\models \neg \varphi$ iff $\sigma_{\mathbb{P}}[i]\models \varphi$ does not hold;

(3) $\sigma_{\mathbb{P}}[i]\models\varphi_{1}\wedge\varphi_{2}$ iff $\sigma_{\mathbb{P}}[i]\models\varphi_{1}$ and $\sigma_{\mathbb{P}}[i]\models\varphi_{2}$;

(4) $\sigma_{\mathbb{P}}[i]\models\varphi_{1}\mathbf{U}\varphi_{2}$ iff there exists $j\geq i$ such that $\sigma_{\mathbb{P}}[j]\models\varphi_{2}$ and for all $k\in\mathbb{N}$ with $i\leq k<j$, we have $\sigma_{\mathbb{P}}[k]\models\varphi_{1}$.

A word $\sigma_{\mathbb{P}}$ satisfies an LTL$_{-X}(\mathbb{P})$ formula $\varphi$, written as $\sigma_{\mathbb{P}}\models\varphi$, if and only if $\sigma_{\mathbb{P}}[1]\models \varphi$.
\end{definition}
\begin{definition}\label{def:traj satis}
Let $T=(Q,A, B,\longrightarrow,O, H)$ be a finite, non-blocking alternating transition system, $\mathbb{P}$ a finite set of atomic propositions and let $\prod:Q\rightarrow 2^{\mathbb{P}}$ be a valuation function. For any  LTL$_{-X}(\mathbb{P})$ formula $\phi$, an infinite sequence $\sigma\in Q^{\omega}$ is said to satisfy $\phi$ w.r.t $\prod$, written as $\sigma\models_{\prod}\phi$, if and only if $\prod(\sigma)\models\phi$, where $\prod(\sigma)\triangleq\prod(\sigma[1])\prod(\sigma[2])\cdots$.
\end{definition}

If the valuation function $\prod$ is known from the context, we often omit the subscript in $\models_{\prod}$.
\section{Kabanza et al.'s algorithm}\label{Sec:strategy}
To solve Problem~\ref{alg}, we will provide a control strategy algorithm based on Kabanza et al.'s planning algorithm.
This section recalls some notions and results about Kabanza et al.'s algorithm.
More details can be found in~\cite{kab:1}.

Kabanza et al.~develop their work in the framework of reactive agent.
Given a finite set $Q$ of world states, a \textit{reactive agent} is described as a pair $(q_{0},succ)$, where $q_{0}\in Q$ is an initial world state and $succ$ is a transition function. For any world state $q\in Q$, $succ(q)$ returns a list $((a_{1},d_{1},W_{1}),\cdots, (a_{n},d_{n},W_{n}))$, where $a_{i}$ is an action that is executable in $q$, $d_{i}$ is a strictly positive real number denoting the duration of $a_{i}$ in $q$, and $W_{i}\subseteq Q$ is the set of nondeterministic successors resulting from the execution of $a_{i}$ in $q$.
As usual, if $q'\in W_{i}$ for some $i\leq n$, then we denote by $q\xrightarrow{a_{i}} q'$ that $q'$ is a successor of $q$ resulting from the execution of  $a_{i}$ in $q$.

\begin{figure}[h]
\begin{center}
\centerline{\includegraphics[scale=0.4]{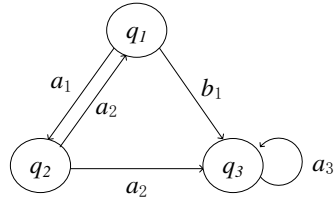}}
\end{center}
\caption{Reactive Agent}\label{Example: reactive agent}
\end{figure}

\begin{example}\label{example:agent}
Fig~\ref{Example: reactive agent} illustrates the reactive agent ($q_{1},succ$) , where $succ(q_{1})=((a_{1}, 1, \{q_{2}\}), (b_{1},1, \{q_{3}\}))$, $succ(q_{2})=((a_{2}, 1, \{q_{1},q_{3}\}))$, and $succ(q_{3})=((a_{3}, 1, \{q_{3}\}))$.
Since the durations of all actions are $1$, we do not indicate them in this figure.
\end{example}

\begin{definition}\label{Def:reactive plan}\cite{kab:1}
A reactive plan is represented by a set of situation control rules (SCRs), where an SCR is a tuple of the form $(n,q,a,N)$ such that:

$\bullet$ $n$ is a number denoting a plan state;

$\bullet$ $q$ is the world state labeling the plan state $n$ and describing the situation when this SCR is applied;

$\bullet$ $a$ is the action to be executed in plan state $n$; and

$\bullet$ $N$ is a set of integers denoting plan states that are nondeterministic successors of $n$ when $a$ is executed \footnote{For any $q'$ with $q\xrightarrow{a}q'$, there must be $j\in N$ such that the corresponding world state of plan state $j$ is $q'$.}.
\end{definition}

In the above definition, two kinds of states are referred to: world states and plan states.
Each plan state is labeled by a world state and different plan states may be labeled by the same world state.
Roughly speaking, these plan states labeled by the same world state $q$ may denote different executive pathes along which the world state $q$ is reached.
So, since the actions to be executed in different plan states may not be identical, the choice of the actions in the world state $q$ can be history dependent.
That is, when $q$ is reached along different pathes, the actions to be executed in $q$ may be different.
Before providing an example to illustrate the above argument, we describe the execution of a reactive plan as follows.

We start the execution of a reactive plan by fetching the SCR corresponding to the initial world state. By convention, this is always the SCR with plan state 1. The corresponding world state describes the current situation before the agent executes any action.
At any time, given the current SCR $(n,q,a,N)$, the action $a$ is executed and the SCR matching the resulting situation is determined from the successor plan states in $N$ by getting an SCR $(n',q',a',N')$ such that $n'\in N$.
In this case, the current situation is $q'$ and then $a'$ is executed.

\begin{figure}[h]
\begin{center}
\centerline{\includegraphics[scale=0.4]{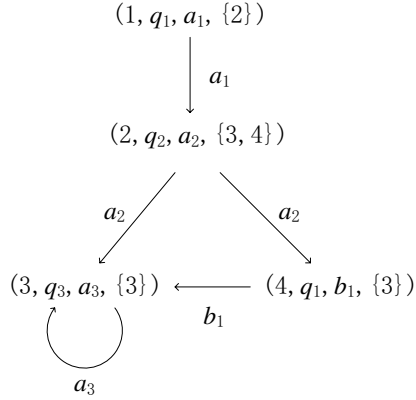}}
\end{center}
\caption{Executing Reactive Plan}\label{Example: reactive plan}
\end{figure}
\begin{example}\label{example:plan}
Consider the reactive agent provided in Example~\ref{example:agent}.
Given a reactive plan
\begin{center}
$RP=\{(1,q_{1},a_{1},\{2\}), (2,q_{2},a_{2},\{3,4\}),(3,q_{3},a_{3},\{3\}),(4,q_{1},b_{1},\{3\})\}$,
 \end{center}
 its execution is illustrated by Fig~\ref{Example: reactive plan}.

In this reactive plan,  both plan states 1 and 4 are labeled by world state $q_1$.
Plan state 1 represents that $q_1$ is the initial state, while plan state 4 means that $q_1$ is reached from $q_2$ by executing $a_2$.
Then it is easy to see that the actions to be executed in $q_1$ may be different  when the pathes along which $q_{1}$ is reached is different.
\end{example}

The trajectory generated by reactive plan is defined as follows.

\begin{definition}\label{def:traj of plan}\cite{kab:1}
Let $(q_{1},succ)$ be a reactive agent and let $RP=\{((1,q_{1},a_{1},N_{1}),\\ (2,q_{2}, a_{2},N_{2}),\cdots (k,q_{k},a_{k},N_{k}))\}$ be a reactive plan of $(q_{1},succ)$.
An infinite sequence $\sigma$ of world states is said to be a trajectory generated by the reactive plan $RP$ if and only if there exists an infinite sequence $\sigma_{N}=i_{1}i_{2}\cdots \in \{1,2,\cdots, k\}^{\omega}$ such that $\sigma_{N}[1]=1$ and for all $j\in\mathbb{N}$, $i_{j+1}\in N_{i_{j}}$ and $q_{i_{j}}=\sigma[j]$.
\end{definition}

\begin{example}\label{example: trajectory}
Consider the reactive agent and the reactive plan $RP$ in Example~\ref{example:agent} and~\ref{example:plan}, respectively.
Let $\sigma_{1}=q_{1}q_{2}q_{3}^{\omega}$ and $\sigma_{2}=q_{1}q_{2}q_{1}q_{3}^{\omega}$.
It is easy to check that $\sigma_{1}$ and $\sigma_{2}$ are exactly trajectories generated by this reactive plan.
\end{example}

\begin{definition}\label{def:satis plan}
Let $\mathbb{P}$ be a finite  set of atomic propositions and let  $\prod$ be a valuation function that assigns each world state $q$ a set $\prod(q)\subseteq\mathbb{P}$. For any LTL$_{-X}(\mathbb{P})$ formula $\phi$, a reactive plan is said to satisfy $\phi$ w.r.t.~$\prod$ if and only if all trajectories generated by this reactive plan satisfy $\phi$ w.r.t.~$\prod$ \footnote{Similar to Definition~\ref{def:traj satis}, we may define the satisfaction relation between LTL$_{-X}(\mathbb{P})$ formulas and trajectories generated by the reactive plan w.r.t.~$\prod$.} and there exists  at least one trajectory generated by this reactive plan.
\end{definition}

\begin{example}\label{Example:satis plan}
Consider the reactive agent and the reactive plan $RP$ in Example~\ref{example:agent} and~\ref{example:plan}, respectively.
Let $\mathbb{P}=\{p_1,p_2,p_3\}$ and let  $\prod:\{q_1,q_2,q_3\}\rightarrow 2^{\mathbb{P}}$ be a valuation function defined as: $\prod(q_1)=\{p_1,p_2\}$, $\prod(q_2)=\{p_2,p_3\}$ and $\prod(q_3)=\{p_1,p_3\}$.
It is easy to check that the reactive plan $RP$ satisfies $p_2\mathbf{U}p_3$ w.r.t. $\prod$.
\end{example}

In \cite{kab:1}, Kabanza et al.~use Metric Temporal Logic (MTL) to specify the desired behaviors of reactive agent.
Given a finite set $\mathbb{P}$ of atomic propositions,
MTL($\mathbb{P}$) formulae are defined as:
\begin{equation*}
\varphi::=p|\neg \varphi |\varphi_{1}\wedge\varphi_{2} | X_{\sim t}\varphi | \Box_{\sim t}\varphi |\varphi_{1}\mathbf{U}_{\sim t}\varphi_{2}
\end{equation*}
where $p\in \mathbb{P}$ is atomic proposition, $X_{\sim t}$, $\Box_{\sim t}$ and $\mathbf{U}_{\sim t}$ are called the \textit{next}, \textit{always} and \textit{until} operators, respectively, $\sim$ denotes either $\leq$, $<$, $\geq$ or $>$, and $t$ is a non-negative real.
Intuitively, if a time constraint "$\sim t$" is associated to a modal operator, then the modal formula connected by this modal operator must hold within a time period satisfying the relation ``$\sim t$".
For example, $\varphi_{1}\mathbf{U}_{\geq t}\varphi_{2}$ means that $\varphi_{1}$ holds until $\varphi_{2}$ becomes true on the semi-open time interval $[t,\infty)$.
So it is easy to see that $\mathbf{U}_{\geq 0}$ coincides with the usual $until$ operator $\mathbf{U}$.
Thus linear temporal logic LTL$_{-X}$($\mathbb{P}$) can be viewed as a sublanguage of MTL($\mathbb{P}$).

Kabanza et al.~also define the semantics of MTL$(\mathbb{P})$.
A careful examination shows that, when we only consider LTL$_{-X}$$(\mathbb{P})$ formulas, Kabanza et al.'s definition is coincided with Definition~\ref{def:satis plan}.
Since the remainder of this paper will mostly refer to  LTL$_{-X}(\mathbb{P})$ formulas, we do not recall the formal definition of the semantics of MTL($\mathbb{P})$. The interested reader may find it in Section 5.2 in \cite{kab:1}.

Kabanza et al.~provide an planning algorithm to construct a reactive plan satisfying an MTL($\mathbb{P}$) formula $\phi$ for the given reactive agent and valuation function $\prod$. The detailed algorithm may be found in \cite{kab:1}.
The following result comes from Theorem 16 and the observation in Section 7.5 in \cite{kab:1}.

\begin{theorem}\label{Th:kab's algorithm o} \cite{kab:1}
Kabanza et al.~planning algorithm is correct and complete. In other words, given a reactive agent $(q_{0},succ)$, an MTL($\mathbb{P}$) formula $\phi$ and a valuation function $\prod$, if Kabanza et al.'s algorithm returns a reactive plan then this reactive plan satisfies $\phi$.
Moreover, Kabanza et al.'s algorithm can find a reactive plan satisfying $\phi$ if such plan exists.
\end{theorem}

Immediately, we have the following corollary, which is trivial but useful.
\begin{corollary}\label{Th:kab's algorithm}
Given a reactive agent $(q_{0},succ)$, an LTL$_{-X}$($\mathbb{P}$) formula $\phi$ and a valuation function $\prod$, if Kabanza et al.'s algorithm returns a reactive plan then this reactive plan satisfies $\phi$.
Moreover, Kabanza et al.'s algorithm can find a reactive plan satisfying $\phi$ if such plan exists.
\end{corollary}
\begin{proof}
Follows from Theorem~\ref{Th:kab's algorithm o} and the fact that linear temporal logic LTL$_{-X}$($\mathbb{P}$) can be viewed as a sublanguage of MTL($\mathbb{P}$). \qquad
\end{proof}
\section{Control strategy algorithm based on Kabanza et al.'s algorithm}
The previous section has provided a brief overview about Kabanza et al.'s planning algorithm.
This section will present a control strategy algorithm based on Kabanza et al.'s algorithm.
Before providing this algorithm, we introduce the notion of control strategy.

\begin{definition}\label{def:control strategy}
Let $T=(Q,A,B,\longrightarrow,O,H)$ be a finite, non-blocking alternating transition system.
For any function $f:Q^{+}\rightarrow A$, we say $f$ is a control strategy of $T$.
For any $q\in Q$ and $f:Q^{+}\rightarrow A$, the outcomes $Out^{n}_{T}(q,f)$ $(n\in\mathbb{N})$ and $Out_{T}(q,f)$  of $f$ from $q$ are defined as follows:
\begin{equation*}
\begin{aligned}Out^{n}_{T}(q,f)=\{s\in Q^{n}: & s[1]=q\ \mathrm{and}\ \forall 1\leq i<n  \exists b_{i}\in B (s[i]\xrightarrow{f(s[1,i]),b_{i}} s[i+1])\},
 \\
 Out_{T}(q,f)=\{\sigma\in Q^{\omega}: & \sigma[1]=q\ \mathrm{and}\ \forall i\in\mathbb{N} \exists b_{i}\in B (\sigma [i]\xrightarrow{ f(\sigma [1,i]),b_{i}} \sigma[i+1])\}.
\end{aligned}
\end{equation*}
Furthermore, we define $Out^{+}_{T}(q,f)$ and $Out^{\infty}_{T}(q,f)$ as:
$Out^{+}_{T}(q,f)=\bigcup_{n\in\mathbb{N}}Out^{n}_{T}(q,f)$ and
$Out^{\infty}_{T}(q,f)=Out^{+}_{T}(q,f)\cup Out_{T}(q,f)$.
\end{definition}

If alternating transition system $T$ is known from the context, we often omit the subscripts in $Out^{n}_{T}(q,f)$, $Out_{T}(q,f)$, $Out^{+}_{T}(q,f)$, and $Out^{\infty}_{T}(q,f)$.

Given a finite, non-blocking alternating transition system $T$, an LTL$_{-X}(\mathbb{P})$ formula~$\phi$ and a valuation function $\prod$, we want to find an initial state $q$ and a control strategy $f$ of $T$ so that $\sigma\models\phi$ for all $\sigma\in  Out(q,f)$.
An algorithm, which is used to find such initial state and control strategy, is presented in Algorithm~\ref{alg} below.

\begin{algorithm}\label{alg}
\caption{Control strategy algorithm}
\SetKwInOut{Input}{input}

(1) \Input{$T$, $\phi$ and $\prod$, where $T=(Q,A,B,\longrightarrow,O,H)$}

(2) Construct a transition function $succ_T$ from $T$

(3) \textbf{for all} $q_{0}\in Q$ \textbf{do}

(4) \ \ \  Adopt Kabanza et al.'s algorithm to find a reactive plan $RP_{kab}$ of $(q_{0}, succ_T)$ enforcing $\phi$ w.r.t. $\prod$

(5) \ \ \  \textbf{if} reactive plan $RP_{kab}$ is found  \textbf{then}

(6) \ \ \ \ \ \ \  $RP$=SimplyReactivePlan($RP_{kab}$)\ \ \  \ \ \ \ /*See Algorithm~\ref{alg:simple} */

(7) \ \ \ \ \ \ \  $f_{RP}$=FunctionStrategy($RP$)\ \ \ \ \  \ \ \ \ \ \ \ \ /*See Algorithm~\ref{alg:strategy} */

(8) \ \ \ \ \ \  Return $q_{0}$ and $f_{RP}$

(9) \ \ \  \textbf{end if}

(10) \textbf{end for}

(11) Return false
\end{algorithm}

In Algorithm~\ref{alg}, steps (2), (6) and (7) are needed to be further refined.
We illustrate them in turn.

\begin{definition}\label{def:transition function}
Let $T=(Q,A,B,\longrightarrow,O,H)$ be a finite, non-blocking alternating transition system and $A=\{a_{1},a_{2},\cdots,a_{k}\}$.
The transition function $succ_{T}$ w.r.t $T$ is defined as: for any $q\in Q$, we set
$succ_{T}(q)=((a_{1},1,W_{1}),(a_{2},1,W_{2}),\cdots,(a_{k},1,W_{k}))$,
 where $W_{i}\triangleq\{q'\in Q:q\xrightarrow{a_{i},b} q'\ \mathrm{for}\ \mathrm{some}\ b\in B\}$ for $i=1,2,\cdots k$.
 \end{definition}

By Definition~\ref{def:system}, for any finite, non-blocking alternating transition system $T=\\ (Q,A,B,\longrightarrow,O,H)$, each set $W_{i}$ mentioned above is finite and non-empty.
Thus for any $q\in Q$, $(q,succ_{T})$ is a reactive agent.
Clearly, due to the finiteness of $Q$, $A$, $B$ and $\rightarrow$, the function $succ_T$ may be obtained using a simple algorithm. We leave it to interested reader.
Before refining steps (6) and (7), we provide some notions and result below.

\begin{definition}
Let $T=(Q,A,B,\longrightarrow,O,H)$ be a finite, non-blocking alternating transition system, $q\in Q$ and let $succ_{T}$ be the transition function w.r.t $T$.
Then any reactive plan of $(q, succ_{T})$ is said to be a reactive plan of $T$.
\end{definition}

\begin{definition}\label{def:cycle of plan}
Let $RP=\{(1,q_{1},a_{1},N_{1}), (2,q_{2},a_{2},N_{2}),\cdots, (k,q_{k},a_{k},N_{k})\}$ be a reactive plan. For any finite sequence $s\in\{1,2,\cdots,k\}^{+}$, if $|s|>1$ and $s[i+1]\in N_{s[i]}$ for all $i<|s|$, then $s$ is said to be a finite path of $RP$.
For any two pathes $s_1$ and $s_2$ of $RP$, if $s_1[1]=1$ and $s_1[end]=s_2[1]=s_2[end]$, then the pair $(s_1, s_2)$ is said to be a reachable cycle of $RP$.
\end{definition}

The following result offers a sufficient and  necessary condition for the existence of trajectory generated by reactive plan.
\begin{lemma}\label{Lem:reactive cycle}
Let $RP=\{(1,q_{1},a_{1},N_{1}), (2,q_{2},a_{2},N_{2}),\cdots, (k,q_{k},a_{k},N_{k})\}$ be a reactive plan. There exists a trajectory generated by $RP$ if and only if there exists a reachable cycle $(s_1,s_2)$ of $RP$.
\end{lemma}
\begin{proof}
(From Right to Left) Let $(s_1,s_2)$ be a reachable cycle of $RP$.
By Definition~\ref{def:cycle of plan}, we have $|s_2|>1$.
Then we set $\sigma_N=s_1\circ (s_2[2,end])^{\omega}$, where $(s_2[2,end])^{\omega}\triangleq s_2[2,end]\circ s_2[2,end]\circ\cdots$.
Since $(s_1,s_2)$ is a reachable cycle of $RP$, it follows from Definition~\ref{def:cycle of plan} that $\sigma_N[1]=1$ and $\sigma_N[i+1]\in N_{\sigma_N[i]}$ for all $i\in\mathbb{N}$.
Then we define an infinite $\sigma\in\{q_1,q_2,\cdots, q_k\}^{\omega}$ as: $\sigma[i]=q_{\sigma_N[i]}$ for all $i\in\mathbb{N}$.
Therefore, since $\sigma_N[1]=1$ and $\sigma_N[i+1]\in N_{\sigma_N[i]}$ for all $i\in\mathbb{N}$, by Definition~\ref{def:traj of plan}, $\sigma$ is generated by $RP$.

(From Left to Right)
Let $\sigma$ be a trajectory generated by $RP$.
Then by Definition~\ref{def:traj of plan}, there exists $\sigma_N\in \{1,2,\cdots,k\}^{\omega}$ such that $\sigma_N[1]=1$ and for all $i\in\mathbb{N}$, $\sigma[i]=q_{\sigma_N[i]}$ and $\sigma_N[i+1]\in N_{\sigma_N[i]}$.
Since the plan state set $\{1,2,\cdots,k\}$ is finite, there exist $j,n\in \mathbb{N}$ such that $1<j<n$ and $\sigma_N[j]=\sigma_N[n]$.
Further, by Definition~\ref{def:cycle of plan}, it is clear that $(\sigma_N[1,j],\sigma_N[j,n])$ is a reachable cycle of $RP$, as desired.
\qquad
\end{proof}

Now we refine steps (6) and (7).
These two steps aim to get a control strategy from a reactive plan.

\textbf{Step (6):} In this step, given a reactive plan $RP$, we will simplify it in this way: for any $(i,q_i,a_i,N_i)$ in $RP$,
if there exist $j_1,j_2,\cdots,j_m \in N_i$ with $m>1$ and $q_{j_1}=q_{j_n}$ for all $n\leq m$, then we remain one of them and remove others from $N_i$.
Thus for any $(i,q_i,a_i,N_i)$ in the simplified reactive plan and for any world state $q$, there exists at most one plan state $j\in N_i$ with $q_j=q$.
Formally, Step (6) is refined in Algorithm~\ref{alg:simple}.

\begin{algorithm}\label{alg:simple}
\caption{Simplifying reactive plan $RP$}
Suppose that $RP=\{(1,q_{1},a_{1},N_{1}), (2,q_{2},a_{2},N_{2}),\cdots, (k,q_{k},a_{k},N_{k})\}$

(1) SimplifyReactivePlan($RP$)\{

(2) $note=0$

(3) \textbf{while} $i\leq k$ and $note=0$ \textbf{do}

(4) \ \ \ \textit{suffix=shortest\_ path(i,i)}

(5) \ \ \ \textbf{if} \textit{suffix}$\neq \emptyset$ \textbf{then}

(6) \ \ \ \ \ \ \textit{prefix=shortest\_ path(1,i)}

(7) \ \ \ \ \ \ \textbf{if} $prefix\neq \emptyset$ \textbf{then}

(8) \ \ \ \ \ \ \ \ \ \ $note=1$;

(9) \ \ \ \ \  \textbf{end if}

(10) \ \  \textbf{end if}

(11) \textbf{end while}

(12) \textbf{for all} $(i,q_i,a_i,N_i)\in RP$

(13) \ \ \ \textbf{for all} $j_1,j_2,\cdots,j_m \in N_i$ with $m>1$ and $q_{j_1}=q_{j_2}=\cdots=q_{j_m}$

(14) \ \ \ \ \ \ \textbf{if} for some $l\leq m$,  there exists $n<|$\textit{prefix}$|$ such that \textit{i=prefix[n]} and \textit{$j_l$=prefix[n+1]} \textbf{then}

(15) \ \ \ \ \ \ \ \ \ $N_i=N_i-\{j_1,\cdots,j_{l-1},j_{l+1},\cdots, j_m\}$ /$*$Remove $j_1,\cdots,j_{l-1},j_{l+1},$ $\cdots, j_m$ from $N_i$ $*$/

(16) \ \ \ \ \ \ \textbf{else if} for some $l\leq m$,  there exists $n<|$\textit{suffix}$|$ such that \textit{i=suffix[n]} and \textit{$j_l$=suffix[n+1]} \textbf{then}

(17) \ \ \ \ \ \ \ \ \ $N_i=N_i-\{j_1,\cdots,j_{l-1},j_{l+1},\cdots, j_m\}$  /$*$Remove $j_1,\cdots,j_{l-1},j_{l+1},$ $\cdots, j_m$ from $N_i$ $*$/

(18) \ \ \ \ \ \ \textbf{else if}

(19) \ \ \ \ \ \ \ \ \ $N_i=N_i-\{j_2,j_3,\cdots, j_m\}$\ \ \ \ \ \ \ \ \  /$*$Remove $j_2,j_3,\cdots, j_m$ from $N_i$ $*$/

(20) \ \ \ \ \ \ \textbf{end if}

(21) \ \ \ \textbf{end for}

(22) \textbf{end for}

(23) Return $RP$\}

\end{algorithm}

In this algorithm, the lines (3)-(11) is used to find a reachable cycle (\textit{prefix},\textit{suffix}).
Amongst, we adopt DijKstra's algorithm~\cite{corman}\cite{dij} to find the shortest pathes of $RP$ from $i$ to $i$ and from $1$ to $i$  (see lines (4) and (6)).
By Lemma~\ref{Lem:reactive cycle} and the completeness of DijKstra's algorithm~\cite{corman}\cite{dij}, \textit{prefix} and \textit{suffix} must can be found in this algorithm if the given reactive plan may generate trajectory.

Suppose that $RP$ may generate trajectory and the reachable cycle (\textit{prefix},\textit{suffix}) has been found.
The lines (12)-(22) aim to simplify the reactive plan $RP$ based \textit{prefix} and \textit{suffix} so that the simplified reactive plan may generate trajectory.
Since \textit{prefix} is the shortest path from 1 to \textit{prefix[end]}, it is clear that there do not exist $i,j<|$\textit{prefix}$|$ such that $i\neq j$ and \textit{prefix[i]=prefix[j]}.
So, for the line (14) in Algorithm~\ref{alg:simple}, there exists at most one natural number $l$ such that $l\leq m$,  \textit{i=prefix[n]} and \textit{$j_l$=prefix[n+1]} for some $n<|$\textit{prefix}$|$.
Similar argument holds for the line (16).
We provide a simple example below to illustrate Algorithm~\ref{alg:simple}.
\begin{example}\label{ex:alg2}
Consider the reactive plan $RP=\{(1,q_1,a_1,\{2\}),(2,q_2,a_2,\{1,4\}),\\ (3,q_3,a_3,\{1\}),(4,q_1,a_4,\{3\})\}$.
We adopt Algorithm~\ref{alg:simple} to simplify $RP$.
It is easy to check that both \textit{suffix} and  \textit{prefix} found in this algorithm are ``$121$''.
For the SCR $(2,q_2,a_2,\{1,4\})\in RP$, since both plan states $1$ and 4 are labeled by $q_1$ and $pre\!fix=121$, plan state 4 is removed from $\{1,4\}$.
One may easily examine that the simplified reactive plan is $\{(1,q_1,a_1,\{2\}),(2,q_2,a_2,\{1\}), (3,q_3,a_3,\{1\}),(4,q_1,a_4,\{3\})\}$.
\end{example}

In the above example, for the plan states 3 and 4 in the simplified reactive plan, there does not exist path from plan state $1$ to these states, although such pathes exist for the original reactive plan.
Thus a natural question arises:  whether the simplification provided in Algorithm~\ref{alg:simple} may result in that the simplified reactive plan can not generate trajectory although the original reactive plan can do so.
The following result reveals that this situation can not arise.
\begin{theorem}\label{th:plan cycle}
Let $RP=\{(1,q_{1},a_{1},N_{1}), (2,q_{2},a_{2},N_{2}),\cdots, (k,q_{k},a_{k},N_{k})\}$ be a reactive plan.
If $RP$ generates trajectory, then so does the simplified reactive plan generated by Algorithm~\ref{alg:simple}.
\end{theorem}
\begin{proof}
Suppose that $RP$ may generate trajectory.
Then, by Lemma~\ref{Lem:reactive cycle} and Algorithm~\ref{alg:simple}, a reachable cycle (\textit{prefix}, \textit{suffix}) of $RP$ must can be found.
Consider the following two cases.

\textbf{Case 1.}
 \textit{prefix[n]}$\neq$\textit{suffix[m]} for any $n<|$\textit{prefix}$|$ and $m<|$\textit{suffix}$|$.
 Then, due to Algorithm~\ref{alg:simple}, it is easy to check that both \textit{prefix} and \textit{suffix} are pathes of the simplified reactive plan.
 Further, since (\textit{prefix},\textit{suffix}) is a reachable cycle of $RP$, by Definition~\ref{def:cycle of plan}, (\textit{prefix},\textit{suffix}) is a reachable cycle of the simplified reactive plan.
Thus by Lemma~\ref{Lem:reactive cycle}, the simplified reactive plan may generate trajectory.

\textbf{Case 2.} \textit{prefix[n]=suffix[m]} for some $n<|$\textit{prefix}$|$ and $m<|$\textit{suffix}$|$.
 Then by Algorithm~\ref{alg:simple}, one may easily examine that both \textit{prefix} and \textit{suffix[1,m]$\circ$prefix[n+1,end]} are pathes of the simplified reactive plan.
 On the other hand, since (\textit{prefix},\textit{suffix}) is a reachable cycle of $RP$, by Definition~\ref{def:cycle of plan}, we get \textit{prefix[end]=suffix[1]=suffix[end]}.
 Then by  Definition~\ref{def:cycle of plan}, (\textit{prefix},\textit{suffix[1,m]$\circ$prefix[n+1,end]}) is a reachable cycle of the simplified reactive plan.
Therefore, by Lemma~\ref{Lem:reactive cycle}, the simplified reactive plan may generate trajectory.
\end{proof}

\begin{theorem}\label{th:sub plan}
Let $T=(Q,A,B,\longrightarrow,O,H)$ be a finite, non-blocking alternating transition system, $\phi$ an LTL$_{-X}(\mathbb{P})$ formula, $\prod$ a valuation function and let $RP=\{(1,q_{1},a_{1},N_{2}), \cdots,(k,q_{k},a_{k},N_{k})\}$ be a reactive plan of~$T$.
We adopt Algorithm~\ref{alg:simple} to simplify $RP$.
Then we have

(1) For any $(i,q_i,a_i,N_i)$ in the simplified reactive plan and for any $q\in Q$, there exists at most one plan state $j\in N_i$ with $q_j=q$.

(2)
If $RP$ satisfies $\phi$ then the simplified reactive plan also satisfies $\phi$.
\end{theorem}
\begin{proof}
(1) holds trivially. We prove (2) below. Clearly, by Algorithm~\ref{alg:simple}, the trajectories generated by the simplified reactive plan can be generated by $RP$.
Therefore, by Theorem~\ref{th:plan cycle} and Definition~\ref{def:satis plan}, the conclusion (2) holds.
\qquad
\end{proof}

\textbf{Step (7).}
Next, we refine Step (7) in Algorithm 1.
In this step, a control strategy will be obtained from the simplified reactive plan.
For this purpose, some result and notion are provided below.
\begin{lemma}\label{lem:plan trajectory}
Let $T=(Q,A,B,\longrightarrow,O,H)$ be a finite, non-blocking alternating transition system and let $RP=\{(1,q_{1},a_{1},N_{2}), \cdots,(k,q_{k},a_{k},N_{k})\}$ be a reactive plan of~$T$.
Suppose that for any $(i,q_i,a_i,N_i)\in RP$ and $q\in Q$, there exists at most one plan state $j\in N_i$ with $q_j=q$.
Then for any $s\in Q^{+}$, there exists at most one path $s_N\in \{1,2,\cdots,k\}^{+}$ such that $|s_N|=|s|$, $s_N[1]=1$ and $s[j]=q_{s_N[j]}$ for all $j\leq |s_N|$.
\end{lemma}
\begin{proof}
Induction on the length of $s$. \qquad
\end{proof}

\begin{definition}\label{def:construct strategy}
Let $T=(Q,A,B,\longrightarrow,O,H)$ be a finite, non-blocking alternating transition system and let $RP=\{(1,q_{1},a_{1},N_{2}), \cdots,(k,q_{k},a_{k},N_{k})\}$ be a reactive plan of~$T$.
Suppose that for any $(i,q_i,a_i,N_i)\in RP$ and state $q\in Q$, there exists at most one plan state $j\in N_i$ with $q_j=q$.
The control strategy $f_{R\!P}:Q^{+}\rightarrow A$ generated by reactive plan $RP$ is defined as: for any $s\in Q^{+}$,
if there exists a path $s_N\in \{1,2,\cdots,k\}^{+}$ such that $|s_N|=|s|$, $s_N[1]=1$ and $s[j]=q_{s_N[j]}$ for all $j\leq |s|$ then we set $f_{R\!P}(s)=a_{s_N[end]}$, otherwise we put $f_{R\!P}(s)=a_{1}$.
\end{definition}

By Lemma~\ref{lem:plan trajectory}, the control strategy $f_{RP}$ defined above is well-defined.
The function $FunctionStrategy(RP)$ in Step (7) in Algorithm~\ref{alg} is capable of producing such control strategy.
The algorithm realizing this function is presented in Algorithm~\ref{alg:strategy}.

\begin{algorithm}\label{alg:strategy}

\caption{Producing control strategy $f_{RP}$}

\SetKwInOut{Input}{input}

Suppose that $RP=\{(1,q_{1},a_{1},N_{1}), (2,q_{2},a_{2},N_{2}),\cdots, (k,q_{k},a_{k},N_{k})\}$

FunctionStrategy($RP$)\{

(1) \Input{$s$    \ \ \ \ \ \ \ \         /*$s$ is an array denoting a sequence of world states*/}

(2) SeqOfPS[1]=1       /*SeqOfPS is an array denoting a sequence of plan states*/

(3) \textbf{if} $s[1]\neq q_{1}$ \textbf{then}

(4) \ \ \ Return $a_1$

(5) \textbf{end if}

(6) $i=2$

(7) \textbf{while} $i\leq |s|$ \textbf{do}

(8) \ \ \ $k=$SeqOfPS$[i-1]$

(9) \ \ \ \textbf{if} $s[i]=q_j$ for some $j\in N_k$ \textbf{then}

(10) \ \ \ \ \ \ SeqOfPS$[i]=j$

(11) \ \ \ \ \ \ $i=i+1$

(12) \ \ \ \textbf{else}

(13) \ \ \ \ \ \ Return $a_1$

(14) \ \ \ \textbf{end if}

(15) \textbf{end while}

(16) $k=$SeqOfPS$[i-1]$

(17) Return $a_k$\}

\end{algorithm}

 Due to the following result, if the simplified reactive plan obtained by performing Algorithm~\ref{alg:simple} satisfies formula $\phi$ then it can generate a control strategy $f_{RP}$ so that $\sigma\models\phi$ for all $\sigma\in Out(q_1,f_{RP})$.
\begin{theorem}\label{th:plan to strategy}
Let $T=(Q,A,B,\longrightarrow,O,H)$ be a finite, non-blocking alternating transition system, $\phi$ an LTL$_{-X}(\mathbb{P})$ formula, $\prod$ a valuation function and let $RP=\{(1,q_{1},a_{1},N_{2}), \cdots,(k,q_{k},a_{k},N_{k})\}$ be a reactive plan of $T$. Suppose that for any $(i,q_i,a_i,N_i)\in RP$ and state $q\in Q$, there exists at most one plan state $j\in N_i$ with $q_j=q$.
 Let $f_{RP}$ be the control strategy generated by $RP$.
Then we have

(1) $Out(q_{1},f_{RP})$ exactly contains trajectories generated by the reactive plan $RP$,

(2) if $RP$ satisfies $\phi$ then $\sigma\models\phi$ for any $\sigma\in Out(q_{1},f_{RP})$.
\end{theorem}
\begin{proof}
By Definition~\ref{def:traj of plan}, \ref{def:control strategy} and~\ref{def:construct strategy}, it is easy to prove (1). Then (2) follows immediately. \qquad
\end{proof}

\begin{corollary}\label{co:algorithm}
Let $T=(Q,A,B,\longrightarrow,O,H)$ be a finite, non-blocking alternating transition system, $\phi$ an LTL$_{-X}(\mathbb{P})$ formula and let $\prod:Q\rightarrow 2^{\mathbb{P}}$ be a valuation function.
If there exists a reactive plan $RP$ of $T$ satisfying $\phi$, then Algorithm~\ref{alg} can find an initial state $q$ and a control strategy $f$ so that  $\sigma\models\phi$ for all $\sigma\in Out(q,f)$.
\end{corollary}
\begin{proof}
Follows from Corollary~\ref{Th:kab's algorithm}, Algorithm~\ref{alg}, Theorem~\ref{th:sub plan} and~\ref{th:plan to strategy}.
\end{proof}

Inspired by Theorem~\ref{th:plan to strategy}, someone may conjecture that given an initial state $q_{0}$ and a control strategy $f$, there exists a reactive plan $RP$ such that $Out(q_{0},f)$ exactly contains trajectories generated by $RP$. This conjecture does not always hold. A counterexample is given below.

\begin{figure}[t]
\begin{center}
\centerline{\includegraphics[scale=0.8]{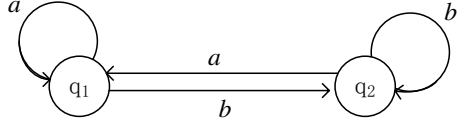}}
\end{center}
\caption{Finite, non-blocking alternating transition system}\label{Example: reactive agent 2}
\end{figure}

\begin{example}\label{Example: not equivalent}
Consider a finite, non-blocking alternating transition system
\begin{center}
$T=(\{q_{1},q_{2}\},\{a,b\},\{1\},\longrightarrow,\{q_{1},q_{2}\},1_{\{q_{1},q_{2}\}})$,
 \end{center}
 where $\longrightarrow$ is described by Fig~\ref{Example: reactive agent 2}.
 Since there only exists one disturbance label, we do not indicate it in this figure.
A control strategy $f:\{q_{1},q_{2}\}^{+}\rightarrow \{a,b\}$ is defined as for any $s\in \{q_{1},q_{2}\}^{+}$,
\begin{equation*}
f(s)=\left\{\begin{aligned}& b   &\textrm{  if $|s|=n(n+3)/2-1$  for some $n\in \mathbb{N}$}\\
& a   &\textrm{otherwise}
\end{aligned}\right.
\end{equation*}
Define a family of finite sequences $s_{k}$ ($k\in\mathbb{N}$) as: $s_{1}=q_{1}q_{2}$ and for any $k>1$, $s_{k}=q_{1}s_{k-1}$.
Let $\sigma=s_{1}s_{2}s_{3}\cdots$.
Thus $\sigma\neq\sigma[1,n](\sigma[n+1,m])^{\omega}$ for any $n,m\in\mathbb{N}$ with $n<m$. It is easy to check that $Out(q_{1},f)=\{\sigma \}$.

Now we show that there does not exist a reactive plan such that $\sigma$ is a trajectory generated by this plan. Suppose that $\sigma$ is generated by the reactive plan $RP=\{(1,q_{1},a_{1},N_{1}), (2,q_{2},a_{2},N_{2}),\cdots, (k,q_{k},a_{k},N_{k})\}$.
Then there exists a sequence $\sigma_{N}=i_{1}i_{2}\cdots$ over $\{1,2,\cdots, k\}$ such that $i_{1}=1$ and for all $j\in\mathbb{N}$, $q_{i_{j}}=\sigma[j]$ and $i_{j+1}\in N_{i_{j}}$. Since $\{1,2,\cdots, k\}$ is a finite set, we have $i_{l}=i_{m}$ for some $l< m$.
On the other hand, since $T$ is determined, we get $N_{i_{j}}=\{i_{j+1}\}$ for all $j\in\mathbb{N}$.
Further, it follows from $i_{l}=i_{m}$ that $i_{l+1}=i_{m+1}$. Similarly, we have $i_{l+j}=i_{m+j}$ for all $j\in\mathbb{N}$. Thus $\sigma_{N}=i_{1}i_{2}\cdots i_{l}\circ(i_{l+1}\cdots i_{m})^{\omega}$ and then $\sigma=q_{i_{1}}q_{i_{2}}\cdots q_{i_{l}}\circ(q_{i_{l+1}}\cdots q_{i_{m}})^{\omega}$.
This contradicts that for any $n,m\in\mathbb{N}$ with $n<m$, $\sigma\neq\sigma[1,n]\sigma[n+1,m]^{\omega}$.  \qquad
\end{example}
\section{Correctness and completeness of control strategy algorithm}
The previous section presents a control strategy algorithm to solve Problem~\ref{problem1}.
This section will deal with its correctness and completeness.
The former is ensured by the result below.

\begin{theorem}\label{th:plan to strategy1}
Given a finite, non-blocking alternating transition system $T=(Q,A,B,\longrightarrow, O,H)$, an LTL$_{-X}(\mathbb{P})$ formula $\phi$ and  a valuation function $\prod:Q\rightarrow 2^{\mathbb{P}}$, if control strategy algorithm returns a state $q_{0}$ and a control strategy $f_{RP}$, then $\sigma\models\phi$ for any $\sigma\in Out(q_{0},f_{RP})$.
\end{theorem}
\begin{proof}
Suppose that control strategy algorithm returns a state $q_{0}$ and a control strategy $f_{RP}$.
Then by Algorithm~\ref{alg}, a reactive plan $RP$ satisfying $\phi$ is found.
Thus by Theorem~\ref{th:sub plan} and~\ref{th:plan to strategy}, we have $\sigma\models\phi$ for any $\sigma\in Out(q_{0},f_{RP})$.
\end{proof}

The rest of this section concerns itself with the completeness of control strategy algorithm.
That is, we consider the following question: given a finite, non-blocking alternating transition system $T$ and an LTL$_{-X}(\mathbb{P})$ formula $\phi$, whether this algorithm must can find an initial state and a control strategy for $T$ enforcing $\phi$ if such state and control strategy exist?
We will provide a partial answer for this question.
Before dealing with this issue, some related notions and results are recalled.

\begin{definition}\label{Def:buchi}
A B$\ddot{u}$chi automaton is a tuple $\mathcal{A}=(S,S_{0},L,\rightarrow_{\mathcal{A}},F)$, where

$\bullet$ $S$ is a finite set of states;

$\bullet$ $S_{0}\subseteq S$ is a set of initial states;

$\bullet$ $L$ is an input  alphabet;

$\bullet$ $\rightarrow_{\mathcal{A}}\subseteq S\times L \times S$ is a transition relation;

$\bullet$ $F\subseteq S$ is a set of accepting states.

An infinite sequence $\sigma \in S^{\omega}$ is said to be a run accepted by $\mathcal{A}$ if and only if $\sigma[1]\in S_{0}$, $\sigma[i]\xrightarrow{a_{i}}_{\mathcal{A}}\sigma[i+1]$ for all $i\in \mathbb{N}$ and there exists $x\in F$ such that $x$ appears infinitely often in $\sigma$.

The  B$\ddot{u}$chi automaton $\mathcal{A}$ is said to be total if both $S_{0}$ and $\{x':x\xrightarrow{l} x'\}$ are singleton sets for any $x\in S$ and $l\in L$.
\end{definition}

\begin{definition}\label{Def:buchi accept}
Let $\mathcal{A}=(S,S_{0},L,\rightarrow_{\mathcal{A}},F)$ be a  B$\ddot{u}$chi automaton.
An infinite sequence $\sigma_{L}\in L^{\omega}$ is accepted by the B$\ddot{u}$chi automaton  $\mathcal{A}$ if and only if there exists a run $\sigma$ accepted by $\mathcal{A}$ such that $\sigma[i]\xrightarrow{\sigma_{L}[i]}_{\mathcal{A}}\sigma[i+1]$ for all $i\in \mathbb{N}$.
\end{definition}

In \cite{Wolper}, it was proven that for any LTL$_{-X}(\mathbb{P}$) formula $\phi$, there exists a B$\ddot{u}$chi automaton $\mathcal{A}_{\phi}$ with input alphabet $2^{\mathbb{P}}$ which accepts exactly the sequences $\sigma\in (2^{\mathbb{P}})^{\omega}$ satisfying formula $\phi$. The interested reader is referred to \cite{Gastin,Gerth,Somenzi,Wolper,Wolper:2} for this topic.

\begin{definition}\label{def:total}
Let $\mathbb{P}$ be a set of atomic propositions. An LTL$_{-X}(\mathbb{P})$ formula $\phi$ is said to be total if there exists a total B$\ddot{u}$chi automaton $\mathcal{A}_{\phi}$ with input alphabet $ 2^{\mathbb{P}}$ such that $\mathcal{A}_{\phi}$ accepts exactly the sequences $\sigma\in (2^{\mathbb{P}})^{\omega}$ satisfying $\phi$.
\end{definition}

Adopting the tool LTL2BA provided by Oddoux and Gastin \cite{Odd}, we may check that the following formulae are total:  $p_{1}\mathbf{U}p_{2}$, $\Box (p_{1}\mathbf{U}p_{2})$, $\diamond (p_{1}\mathbf{U}p_{2})$, $\Box (p_{1}\rightarrow p_{2})$, $\Box\diamond (p_{1}\rightarrow p_{2})$, $\Box (p_{1}\rightarrow \diamond p_{2})$, $\diamond p\wedge\diamond q\wedge\diamond t\wedge\diamond r$, and so on \footnote{The connective $\rightarrow$ and temporal operators $\Box$ and $\diamond$ can be defined as usual, see~\cite{tab:1,Wolper:2}.}.
Some of these formula are considered as control specifications in \cite{fain:2}.

\textbf{Convention.}
\textit{For convenience, for any  total LTL$_{-X}(\mathbb{P})$ formula $\phi$, $\mathcal{A}_{\phi}$ denotes a total B$\ddot{u}$chi automaton with input alphabet $ 2^{\mathbb{P}}$ which accepts exactly the sequences $\sigma\in (2^{\mathbb{P}})^{\omega}$ satisfying $\phi$.}

In the remainder of this section, we will prove that the control strategy algorithm in Algorithm~\ref{alg} is complete w.r.t.~total LTL$_{-X}(\mathbb{P})$ formulae.
Formally, we want to demonstrate that, given a finite, non-blocking alternating transition system $T$, an LTL$_{-X}(\mathbb{P})$ formula $\phi$ and a valuation function $\prod$,
if $\phi$ is total and there exists a state $q_0$ and a control strategy $f_{0}$ so that $\sigma\models\phi\textrm{ for all }\sigma\in Out(q_{0},f_{0})$, then the control strategy algorithm can find an initial state $q$ and a control strategy $f$ of $T$ enforcing $\phi$.
According to Corollary~\ref{co:algorithm}, it is enough to prove that there exists a reactive plan of $T$ satisfying $\phi$.
So in the rest of this section, we will construct such reactive plan.
The desired reactive plan will be obtained from the production automaton of $T$ and $\mathcal{A}_{\phi}$ defined below.
Similar constructions have appeared in \cite{fain:2,kloe:1,tab:1}.

\begin{definition}\label{Def:product automaton}
Let $T=(Q,A,B,{\longrightarrow,} O,H)$ be a finite, non-blocking alternating transition system, $q_{0}\in Q$, $\phi$ a total LTL$_{-X}(\mathbb{P})$ formula, ${\mathcal{A}_{\phi}=(S,\{x_{0}\},2^{\mathbb{P}},\rightarrow_{\mathcal{A}_{\phi}},F)}$ and let
$\prod:Q\rightarrow 2^{\mathbb{P}}$ be a valuation function.
The product automaton of the pair $(T,q_0)$ and $\mathcal{A}_{\phi}$ is defined as $\mathcal{A}_{T,q_0}^{\phi}=(S_{T},S_{T}^{0},A, B,\rightarrow, F_{T})$, where

$\bullet$ $S_{T}=Q\times S$;

$\bullet$ $S_{T}^{0}=\{(q_{0},x_{0})\}$;

$\bullet$ $\rightarrow \subseteq S_{T}\times A\times B \times  S_{T}$ is a transition relation defined as: $(q,x)\xrightarrow{a,b}(q',x')$ if and only if $q\xrightarrow{a,b}q'$ and $x\xrightarrow{\prod(q)}_{\mathcal{A_{\phi}}}x'$;

$\bullet$ $F_{T}=Q\times F$ is a set of accepting states of $\mathcal{A}_{T,q_0}^{\phi}$.

An infinite sequence $\sigma_{T}\in (S_{T})^{\omega}$ is said to be a run accepted by $\mathcal{A}_{T,q_0}^{\phi}$ if and only if the following hold:

(1) $\sigma_{T}[1]\in S_{T}^{0}$,

(2)  for all $i\in \mathbb{N}$, $\sigma_{T}[i]\xrightarrow{a_{i},b_{i}}\sigma_{T}[i+1]$ for some $a_{i}\in A$ and $b_{i}\in B$, and

(3) there exists $(q,x)\in F_{T}$ such that $(q,x)$ appears infinitely often in $\sigma_{T}$.
\end{definition}

It is clear that the sets $S_{T}$ and $F_{T}$ are finite. For any (finite or infinite) sequence $\alpha_{T}=(q_{1},x_{1})(q_{2},x_{2})\cdots$ over $S_{T}$, we define the projections $\Upsilon_{T}(\alpha_{T})=q_{1}q_{2}\cdots$ and $\Upsilon_{A}(\alpha_{T})=x_{1}x_{2}\cdots$.

\begin{lemma}\label{Lem:accepted run of AT}\cite{fain:2,kloe:1}
The projection $\Upsilon_{T}(\sigma_{T})$ of any accepted run $\sigma_{T}$ of $\mathcal{A}_{T,q_0}^{\phi}$ is a trajectory of $T$ satisfying $\phi$.
\end{lemma}

Clearly, for any control strategy $f:Q^{+}\rightarrow A$ of $T$, the function $f_{T}:(S_{T})^{+}\rightarrow A$ defined as $f_{T}\triangleq f\circ \Upsilon_{T}$ is a control strategy of $\mathcal{A}_{T,q_0}^{\phi}$.
The outcome $Out_{\mathcal{A}_{T,q_0}^{\phi}}((q_{0},x_{0}),f_{T})$ of $f_{T}$ from $(q_{0},x_{0})$ is defined as  $Out_{\mathcal{A}_{T,q_0}^{\phi}}((q_{0},x_{0}),f_{T}) \triangleq\{\sigma_{T}\in (S_{T})^{\omega}: \sigma_{T}[1]=(q_{0},x_{0})\ \mathrm{and}\ {\forall i\in\mathbb{N} {\exists b_{i}}\in B} (\sigma_{T}[i]\xrightarrow{f_{T}(\sigma_{T}[1,i]),b_{i}} \sigma_{T}[i+1])\}\}$.
Similarly, we may define $Out_{\mathcal{A}_{T,q_0}^{\phi}}^{n}((q_{0},x_{0}),f_{T})$ ($n\in\mathbb{N}$), $Out_{\mathcal{A}_{T,q_0}^{\phi}}^{+}((q_{0},x_{0}),f_{T})$ and $Out_{\mathcal{A}_{T,q_0}^{\phi}}^{\infty}((q_{0},x_{0}),f_{T})$.
For simplicity, we often omit the subscripts in them.

\begin{lemma}\label{Lem:control strategy for AT}
Let $T=(Q,A,B,{\longrightarrow,} O,H)$ be a finite, non-blocking alternating transition system, $q_{0}\in Q$,
$\phi$ a total LTL$_{-X}(\mathbb{P})$ formula and let  $\prod$ be a valuation function.
Suppose that $\mathcal{A}_{T,q_0}^{\phi}=(S_{T},S_{T}^{0},A, B,\rightarrow, F_{T})$ is the product automaton of the pair $(T,q_0)$ and $\mathcal{A}_{\phi}$ and $f_{0}$ is a control strategy of $T$ so that  $\sigma\models\phi$ for all $\sigma\in Out(q_{0},f_{0})$.
Then, for  control strategy $f_{T}:(S_{T})^{+}\rightarrow A$  with $f_{T}\triangleq f_{0}\circ \Upsilon_{T}$, we have

 (1) $\alpha_{T}\in Out^{\infty}((q_{0},x_{0}),f_{T})$ implies $\Upsilon_{T}(\alpha_{T})\in Out^{\infty}(q_{0},f_{0})$,

 (2) for any $\sigma_{T}\in Out((q_{0},x_{0}),f_{T})$, $\sigma_{T}$ is accepted by $\mathcal{A}_{T,q_0}^{\phi}$.
\end{lemma}
\begin{proof}
Let $f_{T}=f_{0}\circ \Upsilon_{T}$. Then (1) follows from $f_{T}=f_{0}\circ\Upsilon_{T}$, Definition~\ref{Def:product automaton} and the definition of outcomes. Next, we prove (2).
Let $\sigma_{T}\in Out((q_{0},x_{0}),f_{T})$. Then by Definition~\ref{Def:product automaton} and the definition of $Out((q_{0},x_{0}),f_{T})$, it is enough to show that there exists $(q,x)\in F_{T}$ such that $(q,x)$ appears infinitely often in $\sigma_{T}$.
By (1) and $\sigma_{T}\in Out((q_{0},x_{0}),f_{T})$, we obtain $\Upsilon_{T}(\sigma_{T})\in Out(q_{0},f_{0})$. Then since $\sigma\models\phi\textrm{ for all }\sigma\in Out(q_{0},f_{0})$, $\prod(\Upsilon_{T}(\sigma_{T}))$ is accepted by $\mathcal{A}_{\phi}$. Moreover, it follows from Definition \ref{Def:product automaton} that
\begin{equation}\label{Eq:Lem property}
\Upsilon_{A}(\sigma_{T})[i]\xrightarrow{\prod(\Upsilon_{T}(\sigma_{T}))[i]}_{\mathcal{A}_{\phi}}\Upsilon_{A}(\sigma_{T})[i+1] \textrm{ for all } i\in \mathbb{N}.
\end{equation}
Further, since $\mathcal{A}_{\phi}$ is total, $\Upsilon_{A}(\sigma_{T})$ is a unique sequence satisfying (\ref{Eq:Lem property}). Then, since $\prod(\Upsilon_{T}(\sigma_{T}))$ is accepted by $\mathcal{A}_{\phi}$, $\Upsilon_{A}(\sigma_{T})$ is accepted by $\mathcal{A}_{\phi}$.
Thus it follows that there exists $x\in F$ such that $x$ appears infinitely often in $\Upsilon_{A}(\sigma_{T})$.
So, since $T$ is finite, there exists a state $q$ of $T$ such that $(q,x)$ appears infinitely often in $\sigma_{T}$.
\qquad
\end{proof}
\begin{figure}[t]
\begin{center}
\centerline{\includegraphics[scale=0.7]{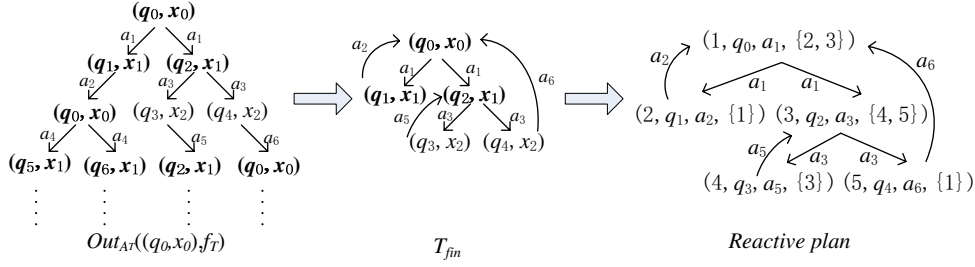}}
\end{center}
\caption{Construction of reactive plan}\label{Fig:illustration}
\end{figure}

In the following, we take two steps to construct the desired reactive plan.
In the first step, we will construct a finite transition transition $T\!_{f\!i\!n}$ based on $Out^{\infty}((q_{0},x_{0}),f_{T})$ such that all trajectories of $T\!_{f\!i\!n}$ are runs accepted by $\mathcal{A}_{T,q_0}^{\phi}$.
In the second step, we may easily obtain a reactive plan from $T\!_{f\!i\!n}$ so that the trajectories generated by this reactive plan are exactly the $\Upsilon_{T}-$projections of trajectories of $T_{f\!i\!n}$.
Then by Lemma~\ref{Lem:accepted run of AT}, this reactive plan satisfies $\phi$.
Fig~\ref{Fig:illustration} illustrates these two steps.
To construct the finite transition transition $T_{f\!i\!n}$, we introduce the following function.

 \begin{definition}\label{Def:number of first cycle}
Let $T=(Q,A,B,{\longrightarrow,} O,H)$ be a finite, non-blocking alternating transition system, $q_{0}\in Q$,
$\phi$ a total LTL$_{-X}(\mathbb{P})$ formula  and let $\prod$ be a valuation function.
Suppose that $\mathcal{A}_{T,q_0}^{\phi}=(S_{T},S_{T}^{0},A, B,\rightarrow, F_{T})$ is the product automaton of the pair $(T,q_0)$ and $\mathcal{A}_{\phi}$, $f_{0}$ is a control strategy of $T$ and
$f_{T}= f_{0}\circ \Upsilon_{T}$.
The function  $ReN: Out^{\infty}((q_{0},x_{0}),f_{T})\rightarrow \mathbb{N}\cup \{\infty\}$ is defined as for any $\alpha_{T}\in Out^{\infty}((q_{0},x_{0}),f_{T})$,
 \begin{equation*}
 ReN(\alpha_{T})=\inf\{n:\textrm{there exist } i< n \textrm{ such that } \alpha_{T}[i]=\alpha_{T}[n]\in F_{T}\}.
 \end{equation*}
 \end{definition}
Here, inf$\emptyset=\infty$.
Intuitively, $ReN(\alpha_{T})<\infty$ means that there exists an accepting state in $F_{T}$ occurring in $\alpha_{T}$ at least two times.
Given a run $\sigma_{T}$ accepted by $\mathcal{A}_{T,q_0}^{\phi}$, by Definition~\ref{Def:product automaton} and~\ref{Def:number of first cycle}, we have $\sigma_{T}[j]=\sigma_{T}[n]\in F_{T}$ for some $j<n$ and then $ReN(\sigma_{T})=n<\infty$.
It is easy to check that $\sigma_{T}[1,j]\circ(\sigma_{T}[j+1,n])^{\omega}$ is also a run accepted by $\mathcal{A}_{T,q_0}^{\phi}$, where $(\sigma_{T}[j+1,n])^{\omega}\triangleq \sigma_{T}[j+1,n]\circ  \sigma_{T}[j+1,n]\circ\cdots$.
Inspired by this fact, we will construct a finite transition transition $T_{f\!i\!n}$ based on $Out^{\infty}((q_{0},x_{0}),f_{T})$ such that the trajectories of $T_{f\!i\!n}$ are runs accepted by $\mathcal{A}_{T,q_0}^{\phi}$.

\begin{definition}\label{def:graph}
Let $T=(Q,A,B,{\longrightarrow,} O,H)$ be a finite, non-blocking alternating transition system, $q_{0}\in Q$,
$\phi$ a total LTL$_{-X}(\mathbb{P})$ formula  and let $\prod$ be a valuation function.
Suppose that $\mathcal{A}_{T,q_0}^{\phi}=(S_{T},S_{T}^{0},A, B,\rightarrow, F_{T})$ is the product automaton of the pair $(T,q_0)$ and $\mathcal{A}_{\phi}$, $f_{0}$ is a control strategy of $T$ and
$f_{T}= f_{0}\circ \Upsilon_{T}$.
The accepting transition system w.r.t. $\mathcal{A}_{T,q_0}^{\phi}$ and $f_{T}$ is defined as
\begin{center}
$T_{fin}(\mathcal{A}_{T,q_0}^{\phi},f_{T})=<S_{f},A, \rightarrow_{f}, lab>$,
\end{center}
where

$\bullet$ $S_{f}=\{s_{T}\in Out^{+}((q_{0},x_{0}),f_{T}): ReN(s_{T})=\infty\}$. That is, the set $S_f$ contains all $s_{T}\in Out^{+}((q_{0},x_{0}),f_{T})$ in which each accepting state occurs at most one time;

$\bullet$ $\rightarrow_{f} \subseteq S_{f}\times A \times  S_{f}$ is a transition relation defined as: $s_{T}\xrightarrow{a}_{f} s_{T}'$ if and only if $a= f_{T}(s_{T})$ and for some $(q,x)\in S_{T}$ and $b\in B$, $s_{T}[end]\xrightarrow{a,b}(q,x)$ and one of the following holds:

(1) $s_{T}\circ(q,x)=s_{T}'$, or

(2) $ReN(s_{T}\circ(q,x))<\infty$, $s_{T}'\prec s_{T}\circ (q,x)$ and $s_{T}'[end]=(q,x)$ \footnote{$s_{T}'\prec s_{T}\circ (q,x)$ means that $s_{T}'$ is a proper prefix of $s_{T}\circ (q,x)$, i.e., $s_{T}\circ (q,x)=s_{T}'s_{T}''$ for some $s_{T}''\in (S_{T})^{+}$.};

$\bullet$ $lab:S_{f}\rightarrow S_{T}$ is a label function defined as: for any $s_{T}\in S_{f}$, $lab(s_{T})=s_{T}[end]$.

An infinite sequence $\sigma_{T}\in (S_{T})^{\omega}$ is said to be a trajectory of $T_{f\!i\!n}(\mathcal{A}_{T,q_0}^{\phi},f_{T})$ if and only if there exists an infinite sequence $s_{T}^{1}s_{T}^{2}\cdots$ over $S_{f}$ such that $s_{T}^{1}=(q_{0},x_{0})$ and for any $i\in\mathbb{N}$, $lab(s_{T}^{i})=\sigma_{T}[i]$ and $s_{T}^{i}\xrightarrow{a_{i}}_{f} s_{T}^{i+1}$.
\end{definition}

The left and middle figures in Fig~\ref{Fig:illustration} illustrate the above construction.
In this figure, the nodes labeled by accepting states of $\mathcal{A}_{T,q_0}^{\phi}$ are identified in boldface type.
In the left figure in Fig~\ref{Fig:illustration}, consider the trajectory $\sigma_T=(q_0,x_0)(q_1,x_1)(q_0,x_0)(q_5,x_1)\cdots$.
Clearly, none of accepting states occurs in $\sigma_T[1]$ or $\sigma_T[1,2]$ two times, while the accepting state $(q_0,x_0)$ occurs in $\sigma_T[1,3]$ two times.
Thus by Definition~\ref{Def:number of first cycle} and~\ref{def:graph}, we have $\sigma_T[1],\sigma_T[1,2]\in S_f$ and $\sigma_T[1,3]\not\in S_f$.
Then
$\sigma_T[1]$ and $\sigma_T[1,2]$ are labeled by $(q_0,x_0)$ and $(q_1,x_1)$, respectively.
Furthermore, by the definition of $\rightarrow_f$, one may check that $\sigma_T[1]\xrightarrow{a_1}_{f}\sigma_T[1,2]$ and $\sigma_T[1,2]\xrightarrow{a_2}_{f}\sigma_T[1]$.

The following result reveals that the state set of $T_{fin}(\mathcal{A}_{T,q_0}^{\phi},f_{T})$ is finite and its trajectories are runs accepted by $\mathcal{A}_{T,q_0}^{\phi}$.

\begin{lemma}\label{Lem:Graph state}
Let $T=(Q,A,B,{\longrightarrow,} O,H)$ be a finite, non-blocking alternating transition system, $q_{0}\in Q$,
$\phi$ a total LTL$_{-X}(\mathbb{P})$ formula and let  $\prod$ be a valuation function.
Suppose that $\mathcal{A}_{T,q_0}^{\phi}$ is the product automaton of the pair $(T,q_0)$ and $\mathcal{A}_{\phi}$ and $f_{0}$ is a control strategy of $T$ so that $\sigma\models\phi$ for all $\sigma\in Out(q_0,f_0)$. Let $f_{T}= f_{0}\circ \Upsilon_{T}$ and let $T_{fin}(\mathcal{A}_{T,q_0}^{\phi},f_{T})=<S_{f}, A, \rightarrow_{f}, lab>$ be the accepting transition system w.r.t. $\mathcal{A}_{T,q_0}^{\phi}$ and $f_{T}$. Then the following conclusions hold:

(1) The set $S_{f}$ is finite and non-empty.

(2) The trajectory $\sigma_{T}$ of $T_{f\!i\!n}(\mathcal{A}_{T,q_0}^{\phi},f_{T})$  is a run accepted by $\mathcal{A}_{T,q_0}^{\phi}$.

(3) For any $s_{T}\in S_{f}$ and for any state $q$ of $T$, if $\Upsilon_{T}(s_{T}[end])\xrightarrow{f_{T}(s_{T}),b}q$ for some $b\in B$, then there exists  $s_{T}'\in S_{f}$ such that $s_{T}\xrightarrow{f_{T}(s_{T})}_{f}s_{T}'$ and ${\Upsilon_{T}(s_{T}')[end]=q}$.
\end{lemma}
\begin{proof}
See Appendix A.
\qquad
\end{proof}

Now we may generate the desired reactive plan from $T_{f\!i\!n}(\mathcal{A}_{T,q_0}^{\phi},f_{T})$.

\begin{definition}\label{def:plan}
Let $T$ be a finite, non-blocking alternating transition system, $q_{0}$ a state $T$,
$\phi$ a total LTL$_{-X}(\mathbb{P})$ formula and let  $\prod$ be a valuation function.
Suppose that $\mathcal{A}_{T,q_0}^{\phi}$ is the product automaton of the pair $(T,q_0)$ and $\mathcal{A}_{\phi}$ and $f_{0}$ is a control strategy of $T$ so that $\sigma\models\phi$ for all $\sigma\in Out(q_0,f_0)$. Let $f_{T}= f_{0}\circ \Upsilon_{T}$ and let $T_{fin}(\mathcal{A}_{T,q_0}^{\phi},f_{T})=<S_{f}, A, \rightarrow_{f}, lab>$ be the accepting transition system w.r.t. $\mathcal{A}_{T,q_0}^{\phi}$ and $f_{T}$ with $S_{f}=\{s^{1}_{T},s^{2}_{T},\cdots, s^{m}_{T}\}$ and $s^{1}_{T}=(q_{0},x_{0})$.
Then the set $RP(T_{fin})$ consists of all SCRs $(i,\Upsilon_{T}(s^{i}_{T}[end]),a_{i},N_{i})$ such that

(1) $1\leq i\leq m$,

(2) $a_{i}=f_{T}(s^{i}_{T})$, and

(3) $N_{i}=\{j\in\mathbb{N}: s^{i}_{T}\xrightarrow{a_{i}}_{f}s^{j}_{T}\}$.
\end{definition}

The right in Fig~\ref{Fig:illustration} illustrates the above construction w.r.t. $T_{f\!i\!n}$ (i.e., the middle one in Fig~\ref{Fig:illustration}).
In this figure, each plan state corresponds to a unique state of $T_{f\!i\!n}$ and the action to be executed in each plan state is set to be the one in the corresponding state of  $T_{f\!i\!n}$.
According to (3) in Lemma~\ref{Lem:Graph state} and Definition~\ref{Def:reactive plan}, $RP(T_{fin})$ defined above is a reactive plan.
In the following, we demonstrate that this reactive plan satisfies $\phi$.

\begin{theorem}\label{Th:plan satisfy}
Let $T=(Q,A,B,{\longrightarrow,} O,H)$ be a finite, non-blocking alternating transition system, $q_{0}\in Q$,
$\phi$ a total LTL$_{-X}(\mathbb{P})$ formula and let  $\prod:Q\rightarrow 2^{\mathbb{P}}$ be a valuation function.
Suppose that $\mathcal{A}_{T,q_0}^{\phi}$ is the product automaton of the pair $(T,q_0)$ and $\mathcal{A}_{\phi}$ and $f_{0}$ is a control strategy of $T$ so that $\sigma\models\phi$ for all $\sigma\in Out(q_0,f_0)$. Let $f_{T}= f_{0}\circ \Upsilon_{T}$, $T_{fin}(\mathcal{A}_{T,q_0}^{\phi},f_{T})$ the accepting transition system w.r.t. $\mathcal{A}_{T,q_0}^{\phi}$ and $f_{T}$ and let $RP(T_{fin})=\{(i,\Upsilon_{T}(s^{i}_{T}[end]),a_{i},N_{i}): 1\leq i\leq m\}$ be the reactive plan defined by Definition~\ref{def:plan}.
 Then for any trajectory $\sigma$ generated by the reactive plan $RP(T_{fin})$, we have $\sigma\models\phi$.
\end{theorem}
\begin{proof}
Let $\sigma$ be a trajectory generated by the reactive plan $RP(T_{fin})$.
So by Definition~\ref{def:traj of plan}, there exists an infinite sequence
$i_{1}i_{2}\cdots $ of plan states in $RP(T_{fin})$ such that
\begin{equation}\label{Eq:Lem plan satisfy}
i_{1}=1,\sigma[j]=\Upsilon_{T}(s^{i_{j}}_{T}[end])\textrm{ and } i_{j+1}\in N_{i_{j}} \textrm{ for all } j\in\mathbb{N}.
\end{equation}
We set $\sigma_{T}=s^{i_{1}}_{T}[end]s^{i_{2}}_{T}[end]\cdots$.
Clearly, $\Upsilon_{T}(\sigma_{T})=\sigma$.
 Therefore, by Lemma~\ref{Lem:accepted run of AT} and~\ref{Lem:Graph state}, in order to prove $\sigma\models\phi$, it suffices to show that $\sigma_{T}$ is a trajectory of $T_{fin}(\mathcal{A}_{T,q_0}^{\phi},f_{T})$.

 It follows from $i_{1}=1$ and Definition~\ref{def:plan} that $s^{i_{1}}_{T}=(q_{0},x_{0})$.
 Let $j\in\mathbb{N}$. By~(\ref{Eq:Lem plan satisfy}), we have $i_{j+1}\in N_{i_{j}}$.
 Further, it follows from Definition~\ref{def:graph} and~\ref{def:plan} that
 $s^{i_{j}}_{T}\xrightarrow{a_{i_{j}}}_{f}s^{i_{j+1}}_{T}$.
 Thus by Definition~\ref{def:graph}, $\sigma_{T}$ is a trajectory of $T_{fin}(\mathcal{A}_{T,q_0}^{\phi},f_{T})$, as desired. \qquad
\end{proof}

Now we arrive at the main result of this section.

\begin{theorem}\label{th:strategy exists}
For any finite, non-blocking alternating transition system $T=(Q,A,B,{\longrightarrow,} O,H)$, LTL$_{-X}(\mathbb{P})$ formula $\phi$ and valuation function $\prod$, if  $\phi$ is total and there exists a state $q$ of $T$ and a control strategy $f:Q^{+}\rightarrow A$ such that $\sigma\models\phi$ for all $\sigma\in Out(q,f)$, then
the control strategy algorithm can find an initial state $q'$ and a control strategy $f':Q^{+}\rightarrow A$ so that $\sigma\models\phi$ for all $\sigma\in Out(q',f')$.
\end{theorem}
\begin{proof}
Let $T=(Q,A,B,\rightarrow,O, H)$ be a finite, non-blocking alternating transition system, $\phi$ an LTL$_{-X}(\mathbb{P})$ formula and $\prod:Q\rightarrow 2^{\mathbb{P}}$ a valuation function.
Suppose that $\phi$ is total and there exists a state $q$ of $T$ and a control strategy $f:Q^{+}\rightarrow A$ such that $\sigma\models\phi$ for all $\sigma\in Out(q,f)$.
Then, by Theorem~\ref{Th:plan satisfy} and Definition~\ref{def:graph} and~\ref{def:plan}, there exists a reactive plan $RP(T_{fin})$ of $T$ such that all trajectories generated by this reactive plan satisfy $\phi$.
Therefore, by Corollary~\ref{co:algorithm}, the control strategy algorithm can find an initial state $q'$ and a control strategy $f_{RP}:Q^{+}\rightarrow A$ so that $\sigma\models\phi$ for all $\sigma\in Out(q',f_{RP})$.
\qquad
\end{proof}
\section{Conclusion and future work}\label{Sec:discussion}
Pola and Tabuada have introduced finite abstractions for control systems $\Sigma$ with disturbance inputs~\cite{pola:5,pola:2}.
However, since these finite abstractions are modeled by finite, non-blocking alternating transition systems rather than usual transition systems, the approaches provided in~\cite{fain:2}\cite{tab:1}\cite{tab:2} are not suitable for finding control strategies for Pola and Tabuada's abstractions.
To overcome this defect, this paper presents a control strategy algorithm based on Kabanza et al.'s planning algorithm (see Algorithm~\ref{alg}).
This control strategy algorithm can be used to find an initial state and a control strategy of finite, non-blocking alternating transition system enforcing an given LTL$_{-X}$ formula.
The correctness and completeness of this algorithm are explored.
We demonstrate that this algorithm is correct (see Theorem~\ref{th:plan to strategy1}) and is complete w.r.t~total  LTL$_{-X}$ formulas (see Theorem~\ref{th:strategy exists}).
But it is still an open problem: whether Theorem~\ref{th:strategy exists} holds for all LTL$_{-X}$ formulas.
We will explore this problem in further work.

Now, we may adopt the control strategy algorithm to find an initial state and a control strategy of Pola and Tabuada's finite abstraction enforcing an LTL$_{-X}$ formula $\phi$.
However, the control problem in the design of control system~is:
\begin{problem}\label{problem2}
Given a control system $\Sigma$ with disturbance inputs and an LTL$_{-X}$ formula $\varphi$ as specification,
how to construct a feedback controller
such that all trajectories of $\Sigma$ with this controller satisfy $\varphi$
even in the presence of disturbance inputs?
\end{problem}

Thus a natural question arises at this point: if an initial state and a control strategy of finite abstraction enforcing an LTL$_{-X}$ formula $\varphi$ have been found, whether the controller for finite abstraction can be applied to the original systems to meet $\varphi$?
We have dealt with this problem in~\cite{jinjin}.
\appendix
\renewcommand\thesection{\appendixname~\Alph{section}}
\renewcommand\theequation{\Alph{section}.\arabic{equation}}
\renewcommand\thesection{A}
\section*{Appendix A}\label{sec:appendix A}
\newtheorem{atheorem}{Theorem}[section]
\newtheorem{alemma}{Lemma}[section]
\newtheorem{adefinition}{Definition}[section]

In this appendix, we fix a finite, non-blocking alternating transition system $T=(Q,A,B,{\longrightarrow,} O,H)$, an initial state $q_{0}\in Q$, a total LTL$_{-X}(\mathbb{P})$ formula $\phi$, $\mathcal{A}_{\phi}=(S,\{x_{0}\},2^{\mathbb{P}},\rightarrow_{\mathcal{A}_{\phi}},F)$, a valuation function $\prod:Q\rightarrow 2^{\mathbb{P}}$, a control strategy $f_{0}:Q^{+}\rightarrow A$ such that $\sigma\models\phi$  for all $\sigma\in Out(q_{0},f_{0})$.
Suppose that ${\mathcal{A}_{T,q_0}^{\phi}=(S_{T},S_T^{0},A, B,\rightarrow, F_{T})}$ is the product automaton of  the pair $(T,q_0)$ and $\mathcal{A}_{\phi}$ (see Definition~\ref{Def:product automaton}), and the control strategy $f_{T}:(S_{T})^{+}\rightarrow A$ is defined as $f_{T}\triangleq f_{0}\circ \Upsilon_{T}$.
Before proving Lemma~\ref{Lem:Graph state}, we provide two auxiliary results.

 \begin{alemma}\label{Lem:the relation between sigma and sigmaT}
 (1) For any $\sigma\in Out(q_{0},f_{0})$, there exists  a unique $\sigma_{T}\in Out((q_{0},x_{0}),f_{T})$ such that $\Upsilon_{T}(\sigma_{T})=\sigma$.

 (2) For any $s\in Out^{+}(q_{0},f_{0})$, there exists a unique $s_{T}\in Out^{+}((q_{0},x_{0}),f_{T})$ such that $\Upsilon_{T}(s_{T})=s$.

 (3) For any $\alpha_{T}\in Out^{\infty}((q_{0},x_{0}),f_{T})$, if $ReN(\alpha_{T})=n$ then  for any $k<n$, $ReN(\alpha_{T}[1,k])=\infty$.
 \end{alemma}
\begin{proof}
(1) Let $\sigma\in Out(q_{0},f_{0})$. Then $\sigma\models\phi$. It follows from Definition~\ref{def:traj satis} that $\prod(\sigma)\models\phi$.
Then $\prod(\sigma)$ is accepted by $\mathcal{A}_{\phi}$.
Thus by Definition~\ref{Def:buchi} and~\ref{Def:buchi accept},  there exists a run $x_{1}x_{2}\cdots\in S^{\omega}$ accepted by $\mathcal{A}_{\phi}$ such that
\begin{equation}\label{Eq:Lem relation of sigma and sigmaT}
x_{1}=x_{0}\textrm{ and }x_{i}\xrightarrow{\prod(\sigma[i])}_{\mathcal{A}_{\phi}} x_{i+1}\ \textrm{for all } i\in\mathbb{N}.
\end{equation}
Moreover, it follows from $\sigma\in Out_{T}(q_{0},f_{0})$ that for any $i\in\mathbb{N}$, there exists $b_{i}\in B$ such that $\sigma[i]\xrightarrow{f_{0}(\sigma[1,i]),b_{i}} \sigma[i+1]$.
This together with (\ref{Eq:Lem relation of sigma and sigmaT}) and Definition~\ref{Def:product automaton} implies that for any $i\in\mathbb{N}$,
\begin{equation}\label{Eq:Lem relation of sigma and sigmaT 2}
(\sigma[i],x_{i})\xrightarrow{f_0(\sigma[1,i]),b_{i}} (\sigma[i+1],x_{i+1}).
\end{equation}

We set $\sigma_{T}=(\sigma[1],x_{1})(\sigma[2],x_{2})\cdots$. Clearly, $\Upsilon_{T}(\sigma_{T})=\sigma$ and $\sigma_{T}[1]=(q_{0},x_{0})$. Furthermore, since $f_{T}=f_0\circ \Upsilon_{T}$, we get $f_{T}(\sigma_{T}[1,i])=f_0(\sigma[1,i])$ for all $i\in\mathbb{N}$.
Thus it follows from (\ref{Eq:Lem relation of sigma and sigmaT 2}) that for any $i\in\mathbb{N}$,
$(\sigma[i],x_{i})\xrightarrow{f_{T}(\sigma_{T}[1,i]),b_{i}} (\sigma[i+1],x_{i+1})$.
Therefore, we obtain $\sigma_{T}\in Out((q_{0},x_{0}),f_{T})$.

To show the uniqueness of such $\sigma_{T}$, let $\sigma_{T} '\in Out((q_{0},x_{0}),f_{T})$ and $\Upsilon_{T}(\sigma_{T}')=\sigma$.
Then since $\mathcal{A}_{\phi}$ is total, there exists a unique run $x_{1}x_{2}\cdots$ such that $x_{1}=x_{0}$ and $x_{i}\xrightarrow{\prod(\sigma[i])}_{\mathcal{A}_{\phi}} x_{i+1}$ for all $i\in\mathbb{N}$.
So by Definition~\ref{Def:product automaton}, it is easy to check that $\Upsilon_{A}(\sigma_{T} ')=\Upsilon_{A}(\sigma_{T})$. Then it follows from $\Upsilon_{T}(\sigma_{T}')=\sigma=\Upsilon_{T}(\sigma_{T})$ that $\sigma_{T}'=\sigma_{T}$.

(2) Let $s\in Out^{+}(q_{0},f_{0})$. Then by the definition of $Out^{+}(q_{0},f_{0})$ and $Out(q_{0},f_{0})$, $s$~is a prefix of $\sigma$ for some $\sigma\in Out(q_{0},f_{0})$.
 So by (1), there exists $\sigma_{T}\in Out((q_{0},x_{0}),f_{T})$ such that $\Upsilon_{T}(\sigma_{T})=\sigma$ and $\sigma_{T}$ is accepted by $\mathcal{A}_{T,q_0}^{\phi}$. Thus we have $\Upsilon_{T}(\sigma_{T}[1,|s|])=s$ and $\sigma_{T}[1,|s|]\in Out^{+}((q_{0},x_{0}),f_{T})$.   Similar to (1), we may show that $\sigma_{T}[1,|s|]$ is a unique sequence satisfying the condition.

(3) Follows  from Definition~\ref{Def:number of first cycle}.
 \qquad
\end{proof}

\begin{alemma}\label{Lem:sigma finite length}
There exists $n\in\mathbb{N}$ such that for all $\sigma_{T}\in Out((q_{0},x_{0}),f_{T})$, we have $ReN(\sigma_{T})\leq n$.
\end{alemma}
\begin{proof}
Suppose that for any $n\in\mathbb{N}$, there exists $\sigma^{n}_{T}\in Out((q_{0},x_{0}),f_{T})$ such that $ReN(\sigma^{n}_{T})> n$. We will give a contradiction. To this end, the following claim is provided first.

 \textbf{Claim}.
We may construct an infinite sequence $\sigma_{T}\in (S_{T})^{\omega}$ satisfying that for any $k\in\mathbb{N}$, there exist $k_{i}\in\mathbb{N} (i\in\mathbb{N})$ with $k_{1}<k_{2}<k_{3}<\cdots$ such that $\sigma^{k_{i}}_{T}[1,k]=\sigma_{T}[1,k]$ for any $i\in\mathbb{N}$.

We construct such  a sequence by induction on $k$. Let $k=1$. We set $\sigma_{T}[1]=(q_{0},x_{0})$ and $k_{i}=i$ for each $i\in\mathbb{N}$. Then for any $i\in\mathbb{N}$, $\sigma^{k_{i}}_{T}[1]=\sigma^{i}_{T}[1]=(q_{0},x_{0})=\sigma_{T}[1]$ follows from $\sigma^{i}_{T}\in Out((q_{0},x_{0}),f_{T})$.

Suppose that $k=m+1$ and we have found $\sigma_T[1,m]$ and $m_{i}\in\mathbb{N} (i\in\mathbb{N})$ with  $m_{1}<m_{2}<m_{3}<\cdots$ such that $\sigma^{m_{i}}_{T}[1,m]=\sigma_{T}[1,m]$ for all $i\in\mathbb{N}$. Since $S_{T}$ is finite, the set $\{\sigma^{m_{i}}_{T}[m+1]:i\in\mathbb{N}\}$ is finite. So there exists $(q_{k},x_{k})\in \{\sigma^{m_{i}}_{T}[m+1]:i\in\mathbb{N}\}$ and $k_{i}\in \{m_{1}, m_{2}, \cdots \} (i\in\mathbb{N})$ with $k_{1}<k_{2}<k_{3}<\cdots$ such that
 $\sigma^{k_{i}}_{T}[m+1]=(q_{k},x_{k})$ for all $i\in\mathbb{N}$.
We set $\sigma_{T}[k]=(q_{k},x_{k})$.
Thus it follows that $\sigma^{k_{i}}_{T}[1,k]=\sigma_{T}[1,m]\circ(q_{k},x_{k})=\sigma_{T}[1,k]$ for all $i\in\mathbb{N}$.
\\

Now, we return to the proof of this lemma. It is easy to check that $\sigma_{T}\in Out((q_{0},x_{0}),f_{T})$.
Then by Lemma~\ref{Lem:control strategy for AT}, $\sigma_{T}$ is accepted by $\mathcal{A}_{T,q_0}^{\phi}$.
To obtain a contradiction, we will show that $\sigma_{T}$ is not accepted by $\mathcal{A}_{T,q_0}^{\phi}$ below.

Let $k\in\mathbb{N}$. Since $k_{1}<k_{2}<\cdots$, there exists $i_{k}\in \{k_{1}, k_{2}, \cdots \}$ such that $i_{k}>k$.
So by the above claim and the supposition at the beginning of the proof, we obtain
$\sigma^{i_{k}}_{T}[1,k]=\sigma_{T}[1,k]$  and $ReN(\sigma^{i_{k}}_{T})>i_{k} >k$.
Further, by Definition~\ref{Def:number of first cycle}, we have $ReN(\sigma_{T})>i_{k} >k$. Then, since $k$ is an arbitrary nature number, we get $ReN(\sigma_{T})=\infty$. Since the accepting state set $F_{T}$ is finite, it follows from Definition~\ref{Def:number of first cycle} and $ReN(\sigma_{T})=\infty$ that there does not exist $(q,x)\in F_{T}$ such that $(q,x)$ appears infinitely often in $\sigma_{T}$. So $\sigma_{T}$ is not accepted by $\mathcal{A}_{T,q_0}^{\phi}$. \qquad
\end{proof}

\textit{
Lemma \ref{Lem:Graph state}.
Let $T=(Q,A,B,{\longrightarrow,} O,H)$ be a finite, non-blocking alternating transition system, $q_{0}\in Q$,
$\phi$ a total LTL$_{-X}(\mathbb{P})$ formula, $\mathcal{A}_{\phi}=(S,\{x_{0}\},2^{\mathbb{P}},\rightarrow_{\mathcal{A}_{\phi}},F)$ and let  $\prod:Q\rightarrow 2^{\mathbb{P}}$ be a valuation function.
Suppose that ${\mathcal{A}_{T,q_0}^{\phi}=(S_{T},S_{T}^{0},A, B,\rightarrow, F_{T})}$ is the product automaton of the pair $(T,q_0)$ and $\mathcal{A}_{\phi}$ and $f_{0}$ is a control strategy of $T$ so that $\sigma\models\phi$ for all $\sigma\in Out(q_0,f_0)$. Let $f_{T}= f_{0}\circ \Upsilon_{T}$ and let $T_{fin}(\mathcal{A}_{T,q_0}^{\phi},f_{T})=<S_{f},\\ A, \rightarrow_{f}, lab>$ be the accepting transition system w.r.t. $\mathcal{A}_{T,q_0}^{\phi}$ and $f_{T}$. Then the following conclusions hold:}

\textit{(1) The set $S_{f}$ is finite and non-empty.}

\textit{(2) The trajectory $\sigma_{T}$ of $T_{fin}(\mathcal{A}_{T,q_0}^{\phi},f_{T})$ is a run accepted by $\mathcal{A}_{T,q_0}^{\phi}$.}

\textit{(3) For any $s_{T}\in S_{f}$ and for any state $q\in Q$ of $T$, if $\Upsilon_{T}(s_{T}[end])\xrightarrow{f_{T}(s_{T}),b}q$ for some $b\in B$, then there exists  $s_{T}'\in S_{f}$ such that $s_{T}\xrightarrow{f_{T}(s_{T})}_{f}s_{T}'$ and ${\Upsilon_{T}(s_{T}')[end]=q}$.}

\begin{proof}
(1) Clearly, $(q_0,x_0)\in S_f$ and then $S_f$ is non-empty. Next, we show that $S_f$ is finite.
By Lemma~\ref{Lem:sigma finite length}, there exists $n\in\mathbb{N}$ such that $ReN(\sigma_{T})\leq n$ for any $\sigma_{T}\in Out((q_{0},x_{0}),f_{T})$.
Since $S_{T}=Q\times S$ is finite, $Out^{i}((q_{0},x_{0}),f_{T})$ is finite for any $i\in\mathbb{N}$ and then $\bigcup_{i< n}Out^{i}((q_{0},x_{0}),f_{T})$ is finite. So to complete the proof, we just need to show that $S_{f}\subseteq \bigcup_{i< n}Out^{i}((q_{0},x_{0}),f_{T})$.

Let $s_{T}\in S_{f}$.
Then by Definition~\ref{def:graph}, we have $ReN(s_{T})=\infty$.
On the other side, by Lemma~\ref{Lem:control strategy for AT}, we obtain $\Upsilon_{T}(s_T)\in Out^{+}(q_0,f_0)$.
Then, since $T$ is non-blocking, by Definition~\ref{def:control strategy}, there exists $\sigma\in Out(q_0,f_0)$ such that $\Upsilon_{T}(s_T)$ is a prefix of $\sigma$.
Thus by Lemma~\ref{Lem:the relation between sigma and sigmaT}, there exists $\sigma_{T}\in Out((q_{0},x_{0}),f_{T})$ such that $s_{T}$ is a prefix of $\sigma_{T}$.
Further, since $ReN(\sigma_{T})\leq n$ and $ReN(s_{T})=\infty$, by Definition~\ref{Def:number of first cycle}, we get
$|s_{T}|<ReN(\sigma_{T})\leq n$.

(2) Let $\sigma_{T}$ be a trajectory of $T_{fin}(\mathcal{A}_{T,q_0}^{\phi},f_{T})$. Then by (2) in Lemma~\ref{Lem:control strategy for AT}, it is enough to show that $\sigma_{T}\in Out((q_{0},x_{0}),f_{T})$.
 By Definition~\ref{def:graph}, there exists a sequence $s_{T}^{1}s_{T}^{2}\cdots$ over $S_{f}$ such that
 \begin{center}
 $s_{T}^{1}=(q_{0},x_{0})$ and for any $i\in\mathbb{N}$, $s_{T}^{i}[end]=\sigma_{T}[i]$ and $s_{T}^{i}\xrightarrow{a_{i}}_{f} s_{T}^{i+1}$.
 \end{center}
 Thus it follows from Definition~\ref{def:graph} that  $\sigma_{T}[1]=(q_{0},x_{0})$ and for any $i\in\mathbb{N}$, there exists $(q,x)\in S_{T}$ and $b\in B$ such that
 $a_{i}= f_{T}(s_{T}^{i})$,  $\sigma_{T}[i]\xrightarrow{a_{i},b}(q,x)$ and $s_{T}^{i+1}[end]=\sigma_{T}[i+1]=(q,x)$.
Then it follows that $\sigma_{T}\in Out((q_{0},x_{0}),f_{T})$.

(3) Let $s_{T}\in S_{f}$, $q\in Q$, $ReN(s_{T})=\infty$ and $\Upsilon_{T}(s_{T}[end])\xrightarrow{f_{T}(s_{T}),b}q$ for some $b\in B$.
For convenience, we put $s=\Upsilon_{T}(s_{T})$.
By (1) in Lemma~\ref{Lem:control strategy for AT}, we have $s\in Out^{+}(q_{0},f_{0})$.
Then it follows from $s[end]\xrightarrow{f_{T}(s_{T}),b}q$ and $f_{T}(s_{T})= f_{0}(s)$ that
$sq\in Out^{+}(q_{0},f_{0})$.
So by (2) in Lemma~\ref{Lem:the relation between sigma and sigmaT}, there exists a unique $s_{T}'\in Out^{+}((q_{0},x_{0}),f_{T})$ such that
$\Upsilon_{T}(s_{T}')=sq$.
Similarly, $s_{T}$ is a unique sequence in $Out^{+}((q_{0},x_{0}),f_{T})$ such that $\Upsilon_{T}(s_{T})=s$.
Thus $s_{T}\circ(q,x)=s_{T}'$ for some state $x\in S$ of $\mathcal{A}_{\phi}$.
If $ReN(s_{T}')=\infty$ then by Definition~\ref{def:graph}, we obtain $s_{T}'\in S_{f}$, $s_{T}\xrightarrow{f_T(s_T)}_{f}s_{T}'$ and $\Upsilon_{T}(s_{T}'[end])=q$.
Suppose that $ReN(s_{T}')<\infty$. Then since $ReN(s_{T})=\infty$ and $s_{T}\circ(q,x)=s_{T}'$, by Definition~\ref{Def:number of first cycle}, there exists $s_{T}''\prec s_{T}'$ such that $s_{T}''[end]=s_{T}'[end]$ and $ReN(s_{T}'')=\infty$.
Further, by Definition~\ref{def:graph}, we have $s_{T}''\in S_{f}$, $s_{T}\xrightarrow{f_T(s_T)}_{f}s_{T}''$ and $\Upsilon_{T}(s_{T}''[end])=q$.
\end{proof}

%
 \clearpage
\end{document}